\documentclass[12pt]{amsart}
\usepackage[latin1]{inputenc}
\usepackage{indentfirst}
\usepackage{amsfonts}
\usepackage{graphicx}
\usepackage{amsmath}
\usepackage{amssymb}
\usepackage{latexsym}
\usepackage{amscd}
\usepackage{xcolor}
\usepackage{float}
\usepackage[pagebackref]{hyperref}
\usepackage[linesnumbered,lined,ruled,commentsnumbered]{algorithm2e}
\usepackage[pagewise,mathlines]{lineno}

\newcommand{\F}{\mathbb {F}}

\newcommand{\Fq}{\mathbb F_q}
\newcommand{\Fqn}{\mathbb F_{q^n}}
\newcommand{\lcm}{\emph{lcm}} 

\newtheorem{theorem}{Theorem}[section]
\newtheorem{definition}[theorem]{Definition}
\newtheorem{lemma}[theorem]{Lemma}
\newtheorem{corollary}[theorem]{Corollary}
\newtheorem{proposition}[theorem]{Proposition}

\newtheorem{observation}[theorem]{Observation}

\begin{document}
%\linenumbers
\title[On arithmetic progressions in finite fields]{On arithmetic progressions in finite fields}
	%of primitive elements in finite fields with one normal term}

%\author{Ab\'ilio Lemos, Victor G.L. Neumann and S\'avio Ribas}
\author[Ab\'ilio Lemos]{Ab\'ilio Lemos}
\address{
Departamento de Matem\'{a}tica\\
Universidade Federal de Vi\c{c}osa (UFV)\\
Vi\c{c}osa, MG\\
36570-000\\
Brazil\\
}
\email{abiliolemos@ufv.br }

\author[Victor G.L. Neumann]{Victor G.L. Neumann}
\address{
Faculdade de Matem\'{a}tica\\
Universidade Federal de Uberl\^andia (UFU)\\
 Uberl\^andia, MG\\
38400-902\\
Brazil\\
}
\email{victor.neumann@ufu.br }

\author[S\'avio Ribas]{S\'avio Ribas}
\address{
Departamento de Matem\'{a}tica\\
Universidade Federal de Ouro Preto (UFOP)\\
Ouro Preto, MG\\
35400-000\\
Brazil\\
}
\email{savio.ribas@ufop.edu.br }

\thanks{The authors were partially supported by FAPEMIG grant RED-00133-21,  the first and third authors were partially supported by FAPEMIG grant APQ-02546-21, and
the second author were partially supported by FAPEMIG grant APQ-03518-18.}

\maketitle

%\vspace{1ex}
%\small{Faculdade de Matem\'{a}tica, Universidade Federal de Uberl\^{a}ndia, 
%	Av. 
%	J. N. 
%	\'{A}vila 2121, 38.408-902 Uberl\^{a}ndia -MG, Brazil }

%\vspace{8ex}
%\noindent
%\textbf{Keywords:} Primitive elements, normal element, arithmetic progression, exponential sums, finite fields.\\
%\noindent
%\textbf{MSC:} 12E20, 11T23

\begin{abstract}
In this paper, we explore the existence of $m$-terms arithmetic progressions in $\mathbb{F}_{q^n}$ with a given common difference whose terms are all primitive elements, and at least one of them is normal. We obtain asymptotic results for $m \ge 4$ and concrete results for $m \in \{2,3\}$, where the complete list of exceptions when the common difference belongs to $\mathbb{F}_{q}^*$ is obtained. The proofs combine character sums, sieve estimates, and computational arguments using the software SageMath.

\vspace{2ex}
%exponential sums.\\
\end{abstract}

\section{Introduction}

Let $q$ be a prime power, and let $n > 1$ be an integer. As usual, $\Fq$ denotes the finite field with $q$ elements, and $\Fqn$ denotes the unique field extension of $\Fq$ of degree $n$. The multiplicative group $\Fqn^*$ is cyclic whose generators are called {\em primitive elements}. We say that $g \in \Fqn$ is {\em normal} if $\{g, g^q, \dots, g^{q^{n-1}}\}$ is an $\Fq$-basis for $\Fqn$ as a vector space. The {\em Primitive Normal Basis Theorem} (see \cite{lenstra}) ensures that a primitive normal element in $\Fqn$ always exists (see \cite{CH} for a computer-free proof). Other problems in this direction concern the existence of a set with some prescribed property whose elements are all primitives (see e.g. \cite{AN, carlitz, cgnt2, Cohen2015, Cohen1985a, Cohen1985b,Cohen1985c, JT, Kapetanakis}).

There is an extensive literature dealing with consecutive primitive elements, and here it is required that the amount of consecutive elements must be at most the characteristic of the finite field. For instance, Cohen (see \cite{Cohen1985a,Cohen1985b,Cohen1985c}) proved that if $q > 7$, then $\Fq$ contains two consecutive primitive elements. Later, it has been proved,  in \cite[Theorem 1]{Cohen2015}, that if $q > 169$ is odd, then there are always three consecutive primitive elements in $\Fq$. Moreover, the precise values $q \le 169$ for which the latter is false are $q=3,$ $5,$ $7,$ $9,$ $13,$ $25,$ $29,$ $61,$ $81,$ $121,$ and $169$. Also in \cite{Cohen2015}, the similar problem with four consecutive primitive elements was studied, and it was conjectured that any finite field with more than 2401 elements contains such consecutive primitive elements. This result was recently settled in \cite{JT}. %\textcolor{green}{In \cite{Booker}, Booker {\it et al.} proved that if $f(x)=ax^2+bx+c \in \Fq[x]$ satisfies $q>211,$ $a\neq0$ and $b^2-4ac\neq0,$ then there exists $\alpha \in \Fq$ such that $\alpha$ and $f(\alpha)$ are both primitive. In \cite{Rani}, Rani {\it et al.} proved that if $q=p^k,$ where $p$ is an odd prime and $f(x)=ax^2+bx+c \in \Fqn[x],$ with $a\neq0$ and $b^2-4ac\neq0,$ then there always exists a primitive normal element $\alpha\in \Fqn$ over $\Fq$ such that $f(\alpha)\in\Fqn$ is also a primitive normal element over $\Fq,$ for all $q\ge 3^{24}$ and $n\in\mathbb{N}.$} 

In this paper, we discuss the existence of $m$-terms arithmetic progressions in $\Fqn$ with a given common difference whose terms are primitive elements, and at least one of them being normal. It is clear that the characteristic of $\Fq$ must be at least $m$. We obtain asymptotic and concrete results, depending on whether $m \ge 4$ or $m \le 3$. In the particular cases where either $m=2$ or $m=3$, we provide the complete list of exceptions when the common difference belongs to $\Fq^*$. In all these cases
the numerical calculations are done using SageMath \cite{SAGE}.
More specifically, we will prove the following results.

\begin{theorem}\label{Teorema1}
Let $q$ be an odd prime power, $m,n \ge 2$, and $\beta \in \Fqn^*$. There exists $\alpha \in \Fqn$ such that the elements $\alpha, \alpha + \beta, \alpha + 2\beta, \dots, \alpha + (m-1)\beta$ are all primitive and at least one of them is normal provided:
\begin{enumerate}
\item [(i)] $m=2,$ except possibly for $(q,n)$ displayed in Table \ref{tablem2qodd} for $q$ odd, and in Table \ref{tablem2qevenfinal} for $q$ even.
\item [(ii)] $m=3,$ except possibly for $(q,n)$ displayed in Table \ref{m3nmaior6} 
%\in \{(3, 7),$ 
%$(7, 7),$
%$(3, 8),$
%$(5, 8),$
%$(7, 8),$
%$(9, 8),$
%$(11, 8),$
%$(13, 8),$
%$(3, 9),$
%$(7, 9),$
%$(3, 10),$
%$(5, 10),$
%$(11, 10),$
%$(3, 12),$
%$(5, 12),$
%$(7, 12),$ 
%$(13, 12),$ 
%$(3, 16)\}$ 
for $n \ge 7$, in addition to $(q,n)$ as given in Lemma \ref{n6m3} for $n=6$, Lemma \ref{n5m3} for $n=5$, Lemma \ref{n4m3} for $n=4$, Lemma \ref{n3m3} for $n=3$, and Lemma \ref{n2m3} for $n=2$. 
\item [(iii)] $m = 4$ and $q^n \geq 3.31 \cdot 10^{2821}.$
%\item [(v)] $m = 5,$ $q^n \geq 3.31 \cdot 10^{2821}$ and $q>4.56\cdot 10^6.$
\item [(iv)] $m \ge 5$ and $q$ sufficiently large.
\end{enumerate}
\end{theorem}

%The case (i) $m=2$ is straightforward from \cite[Theorem 1.4]{Kapetanakis} with $f(x) = x+\beta$. 
The proof relies on character sums and sieve methods that provide an inequality-like condition that guarantees the existence of such arithmetic progressions. The case $m=3$ extends the results of Cohen, Silva and Trudgian on three consecutive primitive elements, as well as the case $m=2$ that extends the results of Cohen on two consecutive primitive elements. For this, we combine further sieving estimates with computational arguments.

For $\beta \in \Fq^*$ and $m=3$, we obtain the complete list of exceptions.

\begin{theorem}\label{Teorema2}
Let $q$ be an odd prime power, $n \ge 2$, and $\beta \in \Fq^*$. There exists $\alpha \in \Fqn$ such that the elements $\alpha, \alpha + \beta, \alpha + 2\beta$ are all primitive and at least one of them is normal except for $(q,n,\beta)$ displayed in Table \ref{table1}.
%\ref{tablex} and 
%$(q,n,\beta)=(9,3,\beta)$ where $\beta$ is any root of the polynomials
%$x^2 +2x+2, x^2 +x+2 \in \F_3[x]$.
%\begin{table}[h]
%\centering
%\begin{tabular}{cc}
%$(q,n)$ & Values of $\beta$ \\
%\hline
%$(3,2)$ & $\beta \in \F_3^*$ \\ \hline 
%$(5,2)$ & $\beta \in \F_5^*$ \\ \hline
%$(7,2)$ & $\beta \in \{\pm 2, \pm 3\}$ \\ \hline
%\end{tabular} \hspace{0.3cm}
%\begin{tabular}{cc}
%	$(q,n)$ & Values of $\beta$ \\
%	\hline
%$(9,2)$ & $\beta \in \F_9^*$ \\ \hline
%$(11,2)$ & $\beta \in \F_{11}^*$ \\ \hline
%	$(13,2)$ & $\beta \in \{\pm 1,\pm 4,\pm 5,\pm 6 \}$ \\ \hline
%\end{tabular}
%\hspace{0.3cm}
%\begin{tabular}{cc}
%	$(q,n)$ & Values of $\beta$ \\
%	\hline
%	$(3,3)$ & $\beta \in \F_{3}^*$ \\ \hline
%	$(3,4)$ & $\beta \in \F_{3}^*$ \\ \hline
%	$(5,4)$ & $\beta \in \F_{5}^*$ \\ \hline 
%\end{tabular}
%\vspace*{0.5cm}
%	\caption{Genuine exceptions for $\beta \in \F_q$}
%	\label{tablex}
%\end{table}
\end{theorem}

As a consequence, we obtain the following result concerning three primitive elements in arithmetic progression. %, which extends the work of Cohen, Oliveira e Silva, and Trudgian \cite{Cohen2015}.

\begin{corollary}\label{3primitivospa}
Let $q$ be an odd prime power, $n \ge 2$, and $\beta \in \Fq^*$. There exists $\alpha \in \Fqn$ such that the elements $\alpha, \alpha + \beta, \alpha + 2\beta$ are all primitive except for the triples $(q,n,\beta)$ displayed in Table \ref{tablecorol1}.
\end{corollary}
The case $\beta=1$ of the corollary above was proved by Cohen, Oliveira e Silva and Trudgian in \cite{Cohen2015}.
%\newpage
For $\beta \in \Fq^*$ and $m=2$, we also obtain the complete list of exceptions.

\begin{theorem}\label{Teorema3}
Let $q$ be a prime power, $n \ge 2$, and $\beta \in \Fq^*$. There exists $\alpha \in \Fqn$ such that the elements $\alpha$ and $\alpha + \beta$ are both primitive and at least one of them is normal except for $(q,n,\beta) = (2,4,1).$
\end{theorem}

In a similar way than Corollary \ref{3primitivospa}, %and extending the work of Cohen \cite{Cohen1985a,Cohen1985b,Cohen1985c}, 
we obtain the following result.

\begin{corollary}\label{2primitivospa}
Let $q$ be a prime power, $n \ge 2$, and $\beta \in \Fq^*$. There exists a primitive element $\alpha \in \Fqn$ such that $\alpha + \beta$ is also primitive.
\end{corollary}

The case $\beta=1$ of the corollary above was proved by Cohen in \cite{Cohen1985a, Cohen1985b, Cohen1985c}.

The paper is organized as follows. In Section \ref{sectionpreli}, we introduce some standard notations and results on arithmetic functions and characters. In Section \ref{sectiongen}, we prove some inequality-like conditions that guarantee the existence of arithmetic progressions formed by primitive elements such that one of its terms is normal, as well as the sieving version of these inequalities. In Section \ref{sectionasymp}, we obtain some asymptotic results, including the proof of Theorem \ref{Teorema1}(iii),(iv). %for $m \ge 3$. 
Finally, in Section \ref{sec:m=3} we present the proof of Theorems \ref{Teorema1}(ii) and \ref{Teorema2}, and in Section \ref{sec:m=2} we present the prove of Theorems \ref{Teorema1}(i) and \ref{Teorema3}, respectively.
%\textcolor{green}{In Section \ref{sectionproof}, we presented the proof of Theorem \ref{Teorema1} for the case $m=3$ and $n\ge2.$ In Section \ref{sectionproof2}, we presented the proof of Theorem \ref{Teorema2}. Finally, in Section \ref{}, we presented the proof of Theorem \ref{Teorema1} for the case $m=2.$}  

\section{Preliminaries}\label{sectionpreli}

In this section, we present some definitions and results that will be used throughout this paper. We refer the reader to \cite{LN} for basic results on finite fields.

For a positive integer $n$, $\varphi(n)$ denotes the Euler totient function and $\mu(n)$ denotes the M\"obius function.
\begin{definition}
\begin{enumerate}
\item[(a)] Let $f(x)\in \mathbb{F}_{q}[x]$. The Euler totient function for polynomials over $\mathbb{F}_q$ is given by
\[\phi(f)= \left| \left( \dfrac{\mathbb{F}_q[x]}{\langle f \rangle} \right)^{*} \right|,\]
where $\langle f \rangle$ is the ideal generated by $f(x)$ in $\mathbb{F}_q[x]$.
\item[(b)] If $t$ is either a positive integer or a monic polynomial over $\mathbb{F}_q$, W(t) denotes the number of squares-free or monic square-free divisors of $t$, respectively.
\item[(c)] If $f(x)\in \mathbb{F}_{q}[x]$ is a monic polynomial, the polynomial M\"obius function $\mu$ is given by $\mu(f)=0$ if $f$ is not square-free, and $\mu(f)=(-1)^r$ if $f$ is a product of $r$ distinct monic irreducible factors over $\mathbb{F}_q$.
\end{enumerate}
\end{definition}

%The multiplicative group $\mathbb{F}_{q^n}^*$ is a $\mathbb{Z}$-module where the action
%is given by $\alpha^r$, for $r\in \mathbb{Z}$ and $\alpha \in \mathbb{F}_{q^n}^*$.
A multiplicative character $\chi$ of $\F_{q^n}^*$ is a group homomorphism of $\F_{q^n}^*$ to $\mathbb{C}^*$. The group of multiplicative characters $\widehat{\mathbb{F}}_{q^n}^*$ becomes a $\mathbb{Z}$-module by defining $\chi^r(\alpha)=\chi(\alpha^r)$ for $\chi \in \widehat{\mathbb{F}}_{q^n}^*$, $\alpha \in \mathbb{F}_{q^n}^*$, and $r \in \mathbb{Z}$. The order of a multiplicative character $\chi$ is the least positive integer $d$ such that $\chi (\alpha)^d=1$ for any $\alpha \in \F_{q^n}^*$.

Let $e$ be a divisor of $q^n-1$. We say that an element $\alpha \in \mathbb{F}_{q^n}^*$ is $e$-free if for every $d \mid e$ such that $d\neq 1,$ there is no element $\beta \in \F_{q^n}$ satisfying $\alpha= \beta^d$. Following \cite[Theorem 13.4.4 e.g.]{galois}, for any $\alpha \in \F_{q^n}^*,$ we get
\begin{equation}\label{eq:omegae}
w_e(\alpha) = \theta(e) \sum_{d \mid e} \frac{\mu(d)}{\varphi(d)} \sum_{(d)} \chi_d(\alpha) =\left\{
\begin{array}{ll}
	1, \quad & \text{if } \alpha \text{ is } e\text{-free}, \\
	0,            & \text{otherwise,}
\end{array}
\right.
\end{equation}
where $\theta(e)=\frac{\varphi(e)}{e}$,
$\displaystyle \sum_{d\mid e}$ runs over all the positive divisors $d$ of $e$,
$\chi_d$ is a multiplicative character of $\F_{q^n}^*$,
and the sum
$\displaystyle \sum_{(d)} \chi_d$ runs over all the multiplicative characters of order $d$.

%\begin{equation}\label{eq:omegae}
%	w_e(\alpha) = \theta(e) \int_{d \mid e} \eta_d(\alpha) =\left\{
%	\begin{array}{ll}
%		1, \quad & \text{if } \alpha \text{ is } e\text{-free}, \\
%		0,            & \text{otherwise,}
%	\end{array}
%	\right.
%\end{equation}
%where $\theta(e)=\frac{\varphi(e)}{e}$,
%$\displaystyle\int_{d \mid e} \eta_d$ denotes the sum
%$\displaystyle \sum_{d \mid e} \frac{\mu(d)}{\varphi(d)} \sum_{(d)} \eta_d$,
%$\eta_d$ is a multiplicative character of $\F_{q^n}$, and the sum
%$\displaystyle \sum_{(d)} \eta_d$ runs over all the multiplicative characters of order $d$.

The additive group $\mathbb{F}_{q^n}$ is an $\mathbb{F}_q[x]$-module where the action is given by $f \circ \alpha =\displaystyle  \sum_{i=0}^r a_i \alpha^{q^i}$, for any $f=\displaystyle \sum_{i=0}^r a_ix^i\in \mathbb{F}_q[x]$ and $\alpha \in \mathbb{F}_{q^n}$. An element $\alpha \in \F_{q^n}$ has $\F_q$-order $h\in \mathbb{F}_q[x]$ if $h$ is the monic polynomial of lowest degree such that $h \circ \alpha=0$. The $\mathbb{F}_q$-order of $\alpha$ will be denoted by $\mathrm{Ord} (\alpha)$, and clearly the $\F_q$-order of an element $\alpha \in \F_{q^n}$ divides $x^n-1$. An additive character $\eta$ of $\F_{q^n}$ is a group homomorphism of $\mathbb{F}_{q^n}$ to $\mathbb{C}^*$. The group of additive characters $\widehat{\mathbb{F}}_{q^n}$ becomes an $\mathbb{F}_q[x]$-module by defining $f\circ \eta (\alpha)=\eta(f \circ \alpha)$ for $\eta \in \widehat{\mathbb{F}}_{q^n}$, $\alpha \in \mathbb{F}_{q^n}$, and $f \in \mathbb{F}_q[x]$.
An additive character $\eta$ has $\F_q$-order $h \in \mathbb{F}_q[x]$ if $h$ is the monic polynomial of smallest degree such that $h \circ \eta = \eta_0$, where $\eta_0$ is the trivial additive character defined by $\eta_0(\alpha) = 1,$ for all $\alpha \in \Fqn$. The $\mathbb{F}_q$-order of $\eta$ will be denoted by $\mathrm{Ord} (\eta)$.

Let $g\in \mathbb{F}_q[x]$ be a  divisor of $x^n-1$. We say that an element $\alpha \in \mathbb{F}_{q^n}$ is $g$-free if for every polynomial $h \in \mathbb{F}_q[x]$  such that $h \mid g$ and $h\neq 1$, there is no element $\beta \in \F_{q^n}$ satisfying $\alpha= h \circ \beta$. As in the multiplicative case, from \cite[Theorem 13.4.4 e.g.]{galois}, for any $\alpha \in \F_{q^n},$ we get
\begin{equation}\label{eq:Omegag}
\Omega_g(\alpha)= \Theta(g) \sum_{h \mid g} \frac{\mu(h)}{\phi(h)} \sum_{(h)}  \eta_h(\alpha) =\left\{
\begin{array}{ll}
	1, \quad & \text{if } \alpha \text{ is } g\text{-free}, \\
	0,            & \text{otherwise,}
\end{array}
\right.
\end{equation}
where 
$\Theta(g)= \frac{\phi(g)}{q^{\deg(g)}}$,
$\displaystyle \sum_{h \mid g}$ runs over all the monic divisors $h\in \mathbb{F}_q[x]$ of $g$,
$\eta_h$ is an additive character of $\F_{q^n}$, and the sum
$\displaystyle \sum_{(h)} \eta_h$ runs over all additive characters of $\mathbb{F}_q$-order $h$.
It is known that there exist $\phi(h)$ of those characters.

%\begin{equation}\label{eq:Omegag}
%	\Omega_g(\alpha)= \Theta(g) \int_{h|g} \chi_h(\alpha) =\left\{
%	\begin{array}{ll}
%		1, \quad & \text{if } \alpha \text{ is } g\text{-free}, \\
%		0,            & \text{otherwise,}
%	\end{array}
%	\right.
%\end{equation}
%where 
%$\Theta(g)= \frac{\Phi_q(g)}{q^{\deg(g)}}$,
%$\displaystyle\int_{h|g} \chi_h$ denotes the sum
%$\displaystyle \sum_{h|g} \frac{\mu_q(h)}{\Phi_q(h)} \sum_{(h)} \chi_h$,
%$\displaystyle \sum_{h|g}$ runs over all the monic divisors $h\in \mathbb{F}_q[x]$ of $g$,
%$\chi_h$ is an additive character of $\F_{q^n}$, and the sum
%$\displaystyle \sum_{(h)} \chi_h$ runs over all additive characters of $\mathbb{F}_q$-order $h$.
%It is known that there exist $\Phi_q(h)$ of those characters.

%\begin{lemma}\label{kernel}\cite[Lemma 2.5.]{AN}
%Let $f,g \in \F_q[x]$ such that $fg=x^n-1$. For
%every $\alpha \in \F_{q^n}$, we have that
%$g \circ \alpha=0$ if and only if $\alpha=f \circ \beta$ for some
%$\beta \in \F_{q^n}$.
%\end{lemma}

%\begin{theorem}\label{equiv-knormal}
%(\cite{knormal}, Theorem 3.2) Let $\alpha \in \mathbb{F}_{q^n}$. The following three properties are equivalent:
%\begin{enumerate}
%\item[(i)] $\alpha$ is $k$-normal over $\F_q$.
%\item[(ii)] Let $V_{\alpha}$ be the $\F_q$-vector space generated by $\{ \alpha, \alpha^q, \ldots, \alpha^{q^{n-1}} \}$, then $\dim V_{\alpha}$ is $n-k$.
%%me parece que no definimos q=modulus, por eso cambie el item ii
%\item[(iii)] $\alpha$ has $\mathbb{F}_q$-order of degree $n-k$.
%\end{enumerate} 
%\end{theorem}

One may check that an element $\alpha \in \mathbb{F}_{q^n}^*$ is primitive if and only if $\alpha$ is $(q^n-1)$-free and $\alpha \in \mathbb{F}_{q^n}$ is normal if and only if $\alpha$ is $(x^n-1)$-free. %Also, we have the following properties:

%\begin{proposition}\label{equiv-freeness}
%(\cite{knormal}, Proposition 5.2)
%Let $m \mid q^n-1$, and denote by $m_0$ the square-free part of $m$ ($m_0$ is the product of the distinct
%prime factors of $m$). Then, for $\alpha \in \F_{q^n}$, the following are equivalent:
%\begin{enumerate}
%\item[(a)] for $r \mid m$, $\alpha$ is $m$-free implies $\alpha$ is $r$-free;
%\item[(b)] for any divisor $r$ of $m$ such that $m_0 \mid r$, $\alpha$ is $r$-free if and only if $\alpha$ is $m$-free;
%\item[(c)] $\alpha$ is $m$-free if and only if $\gcd(m,\frac{q^n-1}{\mathrm{ord}(\alpha)})=1$.
%\end{enumerate}
%\end{proposition}
%
%\begin{proposition}\cite[Theorem 5.4]{knormal}\label{equiv2-freeness}
%Let $x^n - 1$ have factorization into irreducibles given by $f_0^{a_0}  \dots  f_l^{a_l}$ and
% let $\mathrm{Ord}(\alpha) = f_0^{b_0}  \dots  f_l^{b_l}$ for $\alpha \in \mathbb{F}_{q^n}$.
%For any divisor $g$ of $x^n - 1$, the following are equivalent:\\
%(a) $\alpha$ is g-free;\\
%(b) g and $\dfrac{x^n-1}{\mathrm{Ord}(\alpha)}$ are coprime;\\
%(c) if $f_i \mid g$ for some $0\leq i \leq l$, then $f_i^{a_i} \mid \mathrm{Ord}(\alpha)$
%\end{proposition}

The groups $\mathbb{F}_{q^n}^*$ and $\widehat{\mathbb{F}}_{q^n}^*$ are isomorphic as $\mathbb{Z}$-modules, and $\mathbb{F}_{q^n}$ and $\widehat{\mathbb{F}}_{q^n}$ are isomorphic as $\mathbb{F}_q[x]$-modules (see \cite[Theorem 13.4.1.]{galois}).

To finish this section, we present some estimates that are used in the next sections.

\begin{lemma}\cite[Theorem 5.41]{LN}\label{cotamult}
	Let $\chi$ be a multiplicative character of $\mathbb{F}_{q^n}$ of order $r>1$ and $f \in \F_{q^n}[x]$ be a monic polynomial of positive degree such that $f$ is not of the form $g(x)^r$ for some $g \in \mathbb{F}_{q^n}[x]$ with degree at least 1. Let $e$ be the number of distinct roots of $f$ in its splitting field over $\mathbb{F}_{q^n}$. For every $a \in \mathbb{F}_{q^n}$,
	\[
	\left| \sum_{\alpha \in \mathbb{F}_{q^n}} \chi(a f(\alpha)) \right| \leq (e-1)q^{\frac n2}.
	\]
\end{lemma}

An estimate adapted from \cite[Theorem 5.6]{Fu} with $d=1$ will be used  next.
\begin{lemma}  \label{cotaadmult}
	Let $f(x),g(x) \in \F_{q^n}(x)$ be rational functions. Write
	$f(x)=\prod_{j=1}^k f_j(x)^{n_j}$, where
	$f_j(x) \in \F_{q^n}[x]$ are irreducible polynomials  and
	$n_j$ are non-zero integers. Let $D_1=\sum_{j=1}^k \deg (f_j)$,
	let $D_2=\max (\deg (g), 0)$, let $D_3$ be the degree of the denominator of $g(t)$
	and let $D_4$ be the sum of degrees of those irreducible polynomials dividing the denominator of $g$,
	but distinct from $f_j(x)$ ($j=1,\ldots, k$). Let $\chi: \F_{q^n}^\ast \longrightarrow \mathbb{C}^\ast$
	be a multiplicative character of $\F_{q^n}^*$, and let $\eta: \F_{q^n} \longrightarrow \mathbb{C}^\ast$
	be a non-trivial additive character of $\F_{q^n}$. Suppose that
	$g(x)$ is not of the form $r(x)^{q^n}-r(x)$ in $\F(x),$
where $\F$ is the algebraic closure of 	$\F_{q^n}.$
Then we have the estimate
	\[
	\displaystyle \left|
	\sum_{\substack{\alpha \in \mathbb{F}_{q^n}  \\ f(\alpha)\neq 0,\infty , g(\alpha)\neq \infty}} \chi(f(\alpha)) \eta(g(\alpha)) \right|
	\leq
	\left( D_1 + D_2 + D_3 + D_4 - 1  \right) q^{\frac{n}{2}}.
	\]	
\end{lemma}

\section{General results}\label{sectiongen}

We are interested in finding conditions for the existence of $m$-terms arithmetic progressions in $\Fqn$ with a given common difference whose terms are primitive elements, and at least one of them being normal. For this, the following definitions play important roles.

\begin{definition}
For $m \ge 3$, let $N_m$ be the set of pairs $(q,n)$ such that for every $\beta \in \Fqn^*$ there exists $\alpha \in \Fqn$ for which the elements of $A(\alpha,\beta) = \{\alpha+(j-1)\beta \mid 1 \le j \le m\} \subset \Fqn$ are primitive, and at least one of them is normal over $\Fq$.
\end{definition}

\begin{observation}\label{obs:encaixe}
We have $N_m \subset N_{m-1}$. In fact, if the $m$-terms arithmetic progression $\alpha,\alpha+\beta,\dots,\alpha+(m-1)\beta$ comprises primitive elements and one of the terms is normal, then both $\alpha,\alpha+\beta,\dots,\alpha+(m-2)\beta$ and $\alpha+\beta,\alpha+2\beta,\dots,\alpha+(m-1)\beta$ are $(m-1)$-terms arithmetic progressions formed by primitive elements and one of them contains a normal element.
\end{observation}

\begin{definition}\label{def-Nprimnor}
Let $N(e_1,\dots,e_{m},g)$ denote the number of elements $\alpha\in \Fqn$ such that $\alpha+(i-1)\beta$ is $e_i$-free for all $i\in\{1,\dots,m\}$ and there exists $j\in\{1,\dots,m\}$ such that $\alpha+(j-1)\beta$ is $g$-free. %Let $f,g \in \F_q[x]$ be monic divisors of $x^n-1$, with $\deg f =k$,
%and let $m,r \in \mathbb{N}$ be divisors
%of $q^n-1$. 
%We define
%$$
%N_{r,f}(m,g) = \sum_{\alpha \in \mathbb{F}_{q^n}^*}
%\sum_{\beta \in \mathbb{F}_{q^n}}
%w_m(\alpha) \Omega_g(\beta) I_0(\alpha^r - f \circ \beta).
%$$
%From the definition of $w_m$, $\Omega_g$ and Definition \ref{char0},
%$N_{r,f}(m,g)$
%counts the number of pairs $(\alpha,\beta) \in \mathbb{F}_{q^n}^* \times \mathbb{F}_{q^n}$
%such that $\alpha$ is $m$-free, $\beta$ is $g$-free and $\alpha^r=f \circ \beta$.
%In particular, if $N_{r,f}(q^n-1,x^n-1)>0$ then there exists a pair 
%$(\alpha,\beta) \in \mathbb{F}_{q^n}^* \times \mathbb{F}_{q^n}$
%such that 
%$\alpha$ is primitive, $\beta$ is normal and
%$\alpha^r=f \circ \beta$. From Remark \ref{construct-k-r},
%$\alpha^r=f \circ \beta$
%is an $r$-primitive and $k$-normal element of $\mathbb{F}_{q^n}$
%over $\mathbb{F}_q$.
\end{definition}

\begin{definition}\label{def-Njprimnor}
	Let $N_j(e_1,\dots,e_{m},g)$ denote the number of elements $\alpha\in \Fqn$ such that $\alpha+(i-1)\beta$ is $e_i$-free for all $i\in\{1,\dots,m\}$ and $\alpha+(j-1)\beta$ is $g$-free.
\end{definition}

We will shorten $\bar{e} = (e_1,\dots,e_{m})$ and $W(\bar{e}) = \prod_{i=1}^{m}W(e_i).$  When $\bar{e}=\overline{q^n-1},$ we have $e_i=q^n-1$ for all $i\in\{1, \dots,m\}.$ 

\begin{observation}\label{observation}
From Definitions \ref{def-Nprimnor} and \ref{def-Njprimnor} we have 
\[
N(\bar{e},g) \ge
\frac{1}{m} \sum_{j=1}^m  N_j(\bar{e},g).
\]
\end{observation}

The main result of this section is the following sufficient inequality-like condition.

%Using the last definition, we need to find lower estimates for the sum above, in order to guarantee the positivity of $N_{r,f}(q^n-1,x^n-1)$. We have the following result:

%To prove the next theorem, we need the following result. 
%
%\begin{theorem}(\cite[Theorem 5.6]{Fu})\label{tfu}
%Let $f(t),g(t)\in \mathbb{F}_{q^d}(t)$ be rational functions. Write $f(t)=\prod_{j=1}^{k}f_j(t)^{n_j},$ where $f_j(t)\in\mathbb{F}_{q^d}[t]$ are irreducible polynomials and $n_j$ are non-zero integers. Let $D_1=\sum_{j=1}^{k}\deg(f_j),$ let $D_2=\max\{\deg(g),0\},$ let $D_3$ be the degree of the denominator of $g(t)$ and $D_4$ of the sum of the degrees of the irreducible polynomials dividing of the denominator of $g(t),$ but distinct from $f_j(t)$ $(j=1,\dots,k).$ Let $\chi: \mathbb{F}_{q^d}^*\rightarrow \mathbb{\bar{Q}}_\ell^{*}$ be a multiplicative character of $\mathbb{F}_{q^d},$ and let $\psi_q: \mathbb{F}_{q}\rightarrow \mathbb{\bar{Q}}_\ell$ be a non-trivial additive character of $\mathbb{F}_{q}.$ Suppose that $T_{\mathbb{F}_{q^d}/\mathbb{F}_q(g)}=\sum_{i=0}^{d-1}\sigma^{i}(g)$ is not of the form $r^q(t)-r(t)$ in $\mathbb{F}_{q}(t).$ Then we have the estimate
%\begin{equation} \label{d1d2d3d4}
%\left|\displaystyle\sum_{a\in\Fq,f(a)\neq0,g(a)\neq\infty}\chi(f(a))\psi_{q}(T_{\mathbb{F}_{q^d}/\mathbb{F}_q(g(a))})\right|\le (d(D_1+D_3+D_4)+D_2-1)\sqrt{q}.
%\end{equation}     
%
%\end{theorem} 

\begin{theorem}\label{principal}
Let $q$ be a prime power, let $n \ge 2$ be an integer, $e_1,\dots, e_{m}\in \mathbb{N}$ be divisors of $q^n-1$, and $g \in \F_q[x]$ be a monic divisor of $x^n-1.$ If 
\[q^{\frac{n}{2}}\ge m W(g)W(\bar{e}),\]
then $N(\bar{e},g)>0.$ In particular, if 
\[q^{\frac n2} \ge mW(x^n-1)W(\overline{q^n-1}),\]
then $(q,n) \in N_m.$
%$M_f\left(x^n-1,\frac{q^n-1}{r} \right)>M_f(x^n-1,q^n-1)$, 
%$q^{\frac{n}{2}-k} \geq r W(q^n-1)W(\frac{x^n-1}{f})$
%then 
%there exists an $r$-primitive and $k$-normal element of $\mathbb{F}_{q^n}$ over $\mathbb{F}_q$.
\end{theorem}

\begin{proof}
Let $e_1,\dots,e_m \in \mathbb{N}$ be divisors of $q^n-1$, $g \in \mathbb{F}_q[x]$ be a divisor of $x^n-1$, and $\beta \in \Fqn^*$. We will find
a lower bound for $N(\bar{e},g)$.
From Equations \eqref{eq:omegae} and \eqref{eq:Omegag}, Definitions \ref{def-Nprimnor} and \ref{def-Njprimnor},
and Observation \ref{observation}, we have
\[N(\bar{e},g)  \ge \frac{1}{m} \sum_{\alpha\in\Fqn\backslash A} \left[\prod_{i=1}^{m}w_{e_i}(\alpha+(i-1)\beta)\sum_{j=1}^{m}\Omega_g(\alpha+(j-1)\beta)\right],\]
where
\begin{align*}
\sum_{j=1}^{m}\Omega_g(\alpha+(j-1)\beta) &= 
\sum_{j=1}^m \Theta(j) \sum_{h \mid g} \frac{\mu(h)}{\phi(h)} \sum_{(h)} \eta_h(\alpha+(j-1)\beta) \\
&= \Theta(g) \sum_{h \mid g} \frac{\mu(h)}{\phi(h)} \sum_{(h)} \sum_{j=1}^m \eta_h(\alpha+(j-1)\beta), \\
\prod_{i=1}^m w_{e_i}(\alpha + (i-1)\beta) &= \prod_{i=1}^m \left( \theta(e_i) \sum_{d_i \mid e_i} \frac{\mu(d_i)}{\varphi(d_i)} \sum_{(d_i)} \chi_{d_i}(\alpha+(i-1)\beta) \right),
% \\
%&= \prod_{i=1}^m \left( \Theta(e_i) \sum_{\bar d \mid \bar e} \prod_{i=1}^m \frac{\mu(d_i)}{\varphi(d_i)} \sum_{(\bar d)} %\chi_{d_i}(\alpha+(i-1)\beta) \right)
\end{align*}
and $A = \{-(i-1)\beta \mid 1 \le i \le m\}.$
Thus,
\begin{equation}\label{ineq1}
N(\bar e,g) \ge \frac{1}{m} \Theta(g) \theta(\overline{e}) \sum_{\bar d \mid \bar e} \sum_{h \mid g}
\frac{\mu(\bar d)}{\varphi(\bar d)} \frac{\mu(h)}{\phi(h)} \sum_{(\bar d)} \sum_{(h)} S(\chi_{\bar d},\eta_h),
\end{equation}
where $\theta(\bar e)=\prod_{i=1}^m \theta(e_i)$, $\mu(\bar d)=\prod_{i=1}^m \mu(d_i)$,
$\varphi(\bar d)=\prod_{i=1}^m \varphi(d_i)$, $\bar d \mid \bar e$ means $d_i \mid e_i$ for all $i \in \{1,\dots,m\}$, and
\[S(\chi_{\bar d}, \eta_h) = \sum_{\alpha \in \Fqn \backslash A}
\left( 
 \chi_{d_1}(\alpha) \dots \chi_{d_m}(\alpha+(m-1)\beta) \sum_{j=1}^m \eta_h(\alpha+(j-1)\beta)
 \right).\]
We now split into four cases according to possible values of $\bar d$ and $h$.
\begin{itemize}
\item For $\bar d = \bar 1$ and $h=1$, it follows that $S(\chi_{\bar 1},\eta_0) = m(q^n - m)$,
where $\chi_1$ denotes the trivial multiplicative character defined by $\chi_1(\alpha) = 1$ for all $\alpha \in \Fqn^*$, and $\eta_0$ the trivial
additive character.

\item For $\bar d \neq \bar 1$ and $h=1$,
consider $\chi_0$ a generator of the group of multiplicative characters of $\Fqn^*$ (see \cite[Corollary 5.9]{LN}). As consequence, for each $i\in \{1,\ldots, m\}$
there exists an integer $n_i \in \{0,1,\ldots , q^n-2\}$ such that
$\chi_{d_i}(\alpha)=\chi_0(\alpha^{n_i})$ for all $\alpha \in \mathbb{F}_{q^n}^*$. Observe that
$(n_1,\ldots , n_m) \neq \bar{0},$ since $\bar d \neq \bar 1.$
So, we have
\[
|S(\chi_{\bar d},\eta_0)| = m 
\left| \sum_{\alpha \in \Fqn \backslash A} \chi_{d_1}(\alpha) \dots \chi_{d_m}(\alpha+(m-1)\beta)
\right|
 = m \left| 
 \sum_{\alpha \in \Fqn \backslash A}
 \chi_0( F(\alpha) )\right| ,
\]
where $F(\alpha) = \prod_{i=1}^m (\alpha + (i-1)\beta)^{n_i}.$ It is clear that
there is no $G\in\mathbb{F}_{q^n}[x]$ such that $F(x)=G(x)^{q^n-1}$. 
So we may apply
Lemma \ref{cotamult} and
it follows that $|S(\chi_{\bar d},\eta_0)| \le m(m-1)q^{\frac n2}.$

\item For $\bar d = \bar 1$ and $h \neq 1$, 
from \cite[Theorem 5.4]{LN}
it follows that
\[
|S(\chi_{\bar 1},\eta_h)| = \left| \sum_{j=1}^m \eta_h(\beta)^{j-1} \right|
\cdot
%\left|\sum_{\alpha \in \mathbb F_{q^n} \backslash A} \eta_h(\alpha)\right| \le m^2.
\left| - \sum_{\alpha \in A} \eta_h(\alpha)\right| \le m^2.
\]

\item For $\bar d \neq \bar 1$ and $h \neq 1$, it follows that 
\begin{eqnarray*}
\left| S(\chi_{\bar d},\eta_h)\right| & = & \left|\sum_{j=1}^{m}\sum_{\alpha \in \Fqn \backslash A} \prod_{i=1}^{m}\chi_{d_i}(\alpha+(i-1)\beta)\eta_h(\alpha+(j-1)\beta)\right|\\
& \le & \sum_{j=1}^{m}\left|\eta_h((j-1)\beta)\sum_{\alpha \in \Fqn \backslash A} \prod_{i=1}^{m}\chi_{d_i}(\alpha+(i-1)\beta)\eta_h(\alpha)\right|\\
& \le & m\left|\sum_{\alpha \in \Fqn \backslash A} \chi_0(F(\alpha))\eta_h(\alpha)\right|,
\end{eqnarray*}
where $F(\alpha) = \prod_{i=1}^m (\alpha + (i-1)\beta)^{n_i},$ and $(n_1,\ldots , n_m)$ is the set of positive integer we defined before.
Using Lemma \ref{cotaadmult} with $D_1\le m,$ $D_2=1$ and $D_3=D_4=0,$ we have
$\left| S(\chi_{\bar d},\eta_h)\right|\le m^2 q^{\frac n2}.$
\end{itemize}

We may rewrite the right-hand side of Inequality \eqref{ineq1} as
\[
\frac{1}{m} \Theta(g) \theta(\bar e) (S_1+S_2+S_3+S_4),
\]
where 
\[
\begin{array}{ll}
\displaystyle S_1 = S(\chi_{\bar 1},\eta_0) = m(q^n - m),  
&\displaystyle S_2 = \sum_{\substack{\bar d \mid \bar e \\ \bar d \neq \bar 1}}  \frac{\mu(\bar d )}{\varphi(\bar d )}
 \sum_{(\bar d)} S(\chi_{\bar d},\eta_0), \\
\displaystyle S_3 = \sum_{\substack{h \mid g\\ h \ne 1}}  \frac{\mu(h)}{\phi(h)}  \sum_{(h)} S(\chi_{\bar 1},\eta_h), \text { and }  
&\displaystyle S_4 = \sum_{\substack{\bar d \mid \bar e \\ \bar d \neq \bar 1}}
\sum_{\substack{h \mid g\\ h \ne 1}} \frac{\mu(\bar d )}{\varphi(\bar d)} \frac{\mu(h)}{\phi(h)}
\sum_{(\bar d)} \sum_{(h)} S(\chi_{\bar d},\eta_h).
\end{array}
\]

%$S_1=S(\chi_{\bar 1},\eta_0) = m(q^n - m)$,
%\[
%S_2 = \sum_{\substack{\bar d \mid \bar e \\ \bar d \neq \bar 1}}  \frac{\mu(\bar d )}{\varphi(\bar d )}
% \sum_{(\bar d)} S(\chi_{\bar d},\eta_0),
%\,\,\,
%S_3 = \sum_{\substack{h \mid g\\ h \ne 1}}  \frac{\mu(h)}{\phi(h)}  \sum_{(h)} S(\chi_{\bar 1},\eta_h)
%\]
%and
%\[
%S_4 = \sum_{\substack{\bar d \mid \bar e \\ \bar d \neq \bar 1}}
%\sum_{\substack{h \mid g\\ h \ne 1}} \frac{\mu(\bar d )}{\varphi(\bar d)} \frac{\mu(h)}{\phi(h)}
%\sum_{(\bar d)} \sum_{(h)} S(\chi_{\bar d},\eta_h).
%\]
From the considerations above, and
using that there
are $\varphi(d)$ multiplicative characters of order $d$
and $\phi(h)$ additive characters of $\mathbb{F}_q$-order $h$, we get
\begin{eqnarray*}
S_1+S_2+S_3+S_4 & \ge & S_1 - |S_2|-|S_3|-|S_4| \\
& \ge & m(q^n - m) -  m(m-1)q^{\frac n2} (W(\bar{e})-1) \\
&     & -m^2 (W(g)-1) - m^2q^{\frac n2}  (W(g)-1)(W(\bar{e})-1)\\
& >   & mq^n - m^2q^{\frac n2} (W(g)W(\bar{e})-1) + mq^{\frac n2}(W(\bar{e})-1) - m^2 \\
& \ge   & mq^n - m^2q^{\frac n2} (W(g)W(\bar{e})-1),
\end{eqnarray*}
since $m^2\le mq^{\frac n2}$.
Thus, if $q^{\frac n2} \ge m W(g) W(\bar{e}),$ then $N(\bar{e},g)>0$. In particular,
if $q^{\frac n2}\ge m W(x^n-1)W(q^n-1)^m,$ then $N(\overline{q^n-1},x^n-1)  > 0.$

\end{proof}

%\begin{remark}
%Notice that the last result, with $r=k=1$ and $f=x-1$, generalizes
%previous results on the existence of primitive $1$-normal elements (see \cite[Corollary 5.8]{knormal}).
%Also, this is a stronger condition for the existence of primitive, $k$-normal elements 
%than the result given in \cite[Theorem 3.3]{lucas}.
%%which would give better asymptotic results for $q$.
%\end{remark}

The sieving technique from the next two results is similar to others which have appeared in  
previous works about primitive or normal elements.

\begin{lemma}\label{lemmasieve}
Let $q$ be a prime power, $n \ge 2$ be an integer, and $j\in \{1,\ldots,m\}$.
Let $e$ be a divisor of $q^n-1,$ and let $\{p_1,\dots,p_r\}$ be the set of 
all primes which divide $q^n-1$ but do not divide $e.$
Also, let $g \in \F_q[x]$ be a divisor of $x^n -1,$ and let $\{h_1,\ldots,h_s\} \subset \F_q[x]$
be the set of all monic irreducible polynomials which
divide $x^n -1$ but do not divide $g.$ Then
\begin{eqnarray}\label{sieve}
N_j(\overline{q^n-1},x^n-1)& \geq & 
\sum_{i=1}^{r} N_j (p_i e_1,e_2,\ldots ,e_m ,g) + 
\sum_{i=1}^{r} N_j (e_1,p_i e_2,e_3,\ldots ,e_m ,g) \nonumber\\
& & + \dots + 
\sum_{i=1}^{r}N_j(e_1,\ldots , e_{m-1}, p_i e_m,g)
+ \sum_{i=1}^{s}N_j(\bar{e}, h_i g) \nonumber \\
& & 
- (mr+s-1)N_j(\bar{e},g). 
\end{eqnarray}
\end{lemma}
\begin{proof}
The left-hand side of \eqref{sieve} counts every $\alpha \in \mathbb{F}_{q^n}$ for which
$\alpha + (i-1)\beta$ is primitive for every $i \in \{1,\ldots , m\}$ and
$\alpha + (j-1)\beta$ is normal. Observe that if $\alpha$ is one of these elements,
then  $\alpha + (i-1)\beta$ is $e_i$-free, $p_te_i$-free for all 
$i \in \{ 1,\ldots ,m\}$ and all $t \in \{1,\ldots , r \}$, and
$\alpha + (j-1)\beta$ is $g$-free and $h_ig$-free for all
$i \in \{ 1,\ldots ,s\}$, so $\alpha$ is counted $(mr+s)-(mr+s-1)=1$ times on the 
right-hand side of \eqref{sieve}.
For any other $\alpha \in \mathbb{F}_{q^n}$, we
have that either $\alpha + (i-1)\beta$ is not $p_t e_i$-free
for some  $i \in \{1,\ldots , m\}$ and some $t \in \{1,\ldots , r \}$,
or $\alpha + (j-1)\beta$ is not $h_ig$-free for some $i \in \{ 1,\ldots ,s\}$,
so
$\alpha$ will not be counted in at least one of the
first $m+1$
sums of the right-hand side of \eqref{sieve}.

\end{proof}

\begin{proposition}\label{sieve-prop}
Let $q$ be a prime power, and let $n \ge 2$ be an integer.
Let $e$ be a divisor of $q^n-1,$ and let $\{p_1,\dots,p_r\}$ be the set of 
all primes which divide $q^n-1$ but do not divide $e.$
Also, let $g \in \F_q[x]$ be a divisor of $x^n -1,$ and $\{h_1,\dots,h_s\} \subset \F_q[x]$
be the set of all monic irreducible polynomials which divide $x^n -1$ but do not divide $g.$
Suppose that 
$\delta=1-m\sum_{i=1}^{r}\frac{1}{p_i}-\sum_{i=1}^{s}\frac{1}{q^{\deg h_i}}>0$
and let
$\Delta=2+ \frac{mr+s-1}{\delta}$. 
If
$q^{\frac{n}{2}} \geq m \Delta W(g) W(e)^m,$
then $N(\overline{q^n-1},x^n-1)  > 0.$
\end{proposition}
\begin{proof}
We can rewrite Inequality \eqref{sieve} in the form
\begin{eqnarray*}
N_j(\overline{q^n-1},x^n-1)& \geq & 
\sum_{i=1}^{r} \left[ N_j(p_i e_1,e_2,\ldots ,e_m ,g) - \theta(p_i)N_j(\bar{e}, g)\right] \\
& &
+ \sum_{i=1}^{r} \left[ N_j (e_1,p_i e_2,e_3,\ldots ,e_m ,g) - \theta(p_i)N_j(\bar{e}, g)\right] \\
& & + \dots + \\
& & 
+ \sum_{i=1}^{r} \left[ N_j(e_1,\ldots , e_{m-1}, p_i e_m,g) - \theta(p_i)N_j(\bar{e}, g)\right] \\
& & + \sum_{i=1}^{s}\left[ N_j(\bar{e}, h_i g) - \Theta(h_i)N_j(\bar{e}, g)\right] + \delta  N_j(\bar{e},g),
\end{eqnarray*}
where $e=e_1=\dots = e_m$, $\theta(e) = \frac{\varphi(e)}{e}$, and $\Theta(g) = \frac{\phi(g)}{q^{\deg(g)}}$.

Let $i \in \{1,\ldots,r\}$. From the definitions of $w_m$, $\Omega_g$ and 
Definition \ref{def-Njprimnor}, and
taking into account that
$\theta$ is a multiplicative function, we get
\begin{eqnarray*}
N_j(p_ie_1,e_2,\ldots ,e_m,g) & = &  \Theta(g) \theta(p_i) \theta(\bar e)
\sum_{\substack{d_1 \mid p_i e_1 \\ d_t \mid e_t\\ t\in \{2,\ldots,m \}}}
\sum_{h \mid g} \frac{\mu(\bar d)}{\varphi(\bar d)} \frac{\mu(h)}{\phi(h)} \sum_{(\bar d)} \sum_{(h)}
S_j(\chi_{\bar d},\eta_h) \\
& = & 
\theta(p_i)N_j(\bar e,g) \\ & &
+ \Theta(g) \theta(p_i) \theta(\bar e)
\sum_{\substack{d_1 \mid p_i e_1 \\ p_i \mid d_1 \\ d_t \mid e_t\\ t\in \{2,\ldots,m \}}}
\sum_{h \mid g} \frac{\mu(\bar d)}{\varphi(\bar d)} \frac{\mu(h)}{\phi(h)} \sum_{(\bar d)} \sum_{(h)}
S_j(\chi_{\bar d},\eta_h) ,
\end{eqnarray*}
where
\[S_j(\chi_{\bar d}, \eta_h) = \sum_{\alpha \in \Fqn \backslash A}
\chi_{d_1}(\alpha) \dots \chi_{d_m}(\alpha+(m-1)\beta)  \eta_h(\alpha+(j-1)\beta).\]

Now, from Lemmas \ref{cotamult} and \ref{cotaadmult}
we have $|S_j(\chi_{\bar d}, \eta_h)| \le m q^{\frac n2}$. So
\[
|N_j(p_ie_1,e_2,\ldots , e_m,g) - \theta(p_i)N_j(\bar e,g)| \le \Theta(g) \theta(p_i) \theta(\bar e) m q^{\frac n2} W(g)W(\bar e).
\]
In a similar way, for any $t \in \{ 1,\ldots , m\}$ we have
\[
|N_j(e_1,\ldots p_i e_t , \ldots ,e_m,g) - \theta(p_i)N_j(\bar e,g)| \le \Theta(g) \theta(p_i) \theta(\bar e) m q^{\frac n2}
W(g)W(\bar e).
\]

Let $i \in \{1,\ldots,s\}$.
Once again, from the definitions of $w_m$, $\Omega_g$ and 
Definition \ref{def-Njprimnor}, and
taking into account that
$\Theta$ is a multiplicative function, we have
\begin{eqnarray*}
N_j(\bar e, h_ig) & = &  \Theta(h_i)\Theta(g)  \theta(\bar e)
	\sum_{\bar d \mid \bar e}
		\sum_{h \mid h_i g}  \frac{\mu(\bar d)}{\varphi(\bar d)} \frac{\mu(h)}{\phi(h)} \sum_{(\bar d)} \sum_{(h)}
	S_j(\chi_{\bar d},\eta_h) \\
	& = & 
	\Theta(h_i)N_j(\bar e, h_ig) \\ & &
	+ \Theta(h_i) \Theta(g)  \theta(\bar e)
	\sum_{\bar d \mid \bar e}
\sum_{\substack{h \mid h_ig \\ h_i \mid h}} \frac{\mu(\bar d)}{\varphi(\bar d)} \frac{\mu(h)}{\phi(h)} \sum_{(\bar d)} \sum_{(h)}
	S_j(\chi_{\bar d},\eta_h) ,
\end{eqnarray*}
Now, from Lemma \ref{cotaadmult}
we have $|S_j(\chi_{\bar d}, \eta_h)| \le m q^{\frac n2}$. So
\[
|N_j(\bar e,h_ i g) - \Theta(h_i)N_j(\bar e,g)| \le \Theta(h_i) \Theta(g) \theta(\bar e) m q^{\frac n2} W(g)W(\bar e).
\]

Combining all previous inequalities we obtain
\begin{align*}
N_j(\overline{q^n-1},x^n-1) \geq & \;\delta  N_j(\bar{e},g) \\
		&  - \Theta(g) \theta(\bar e)W(g)W(\bar e) m q^{\frac n2}  \left( m \sum_{i=1}^r \theta(p_i) + \sum_{i=1}^{s}\Theta(h_i) \right).
\end{align*}
Therefore, following the ideas 
from the proof of Theorem \ref{principal}, we have
\[
N_j(\bar{e},g) > \Theta(g) \theta(\bar e) (q^n - mq^{\frac n2}W(g)W(\bar e) ),
\]
and
\begin{align*}
N_j(\overline{q^n-1},x^n-1) & > \Theta(g) \theta(\bar e) q^{\frac n2} \Bigg[\delta ( q^{\frac n2} - mW(g)W(\bar e)) \\
	&  - m W(g)W(\bar e)  \left( m \sum_{i=1}^r \theta(p_i) + \sum_{i=1}^{s}\Theta(h_i) \right) \Bigg] \\
	& = \delta \Theta(g) \theta(\bar e) q^{\frac n2} \left(q^{\frac n2} - m\Delta W(g)W(\bar e) \right).
\end{align*}
Substituing $W(\bar e)=W(e)^m$ and from Observation \ref{observation}, we obtain the desired result.

\end{proof}

%\begin{proposition}\label{caseall} %Versaoum
%Let $n$ be a positive integer and let $q$ be a prime power.
%Let $r,k \in \mathbb{N}$ such that $r$ is a divisor
%of $q^n-1$, $k<n/2$,
%there exists a degree $k$ factor of $x^n-1$ in $\mathbb{F}_q[x]$ and
%$(n-k)^2 \leq q$. 
%If $q^{\frac{n}{2}-k} \geq r(n-k+2) W(q^n-1)$, then 
%there exists an $r$-primitive $k$-normal element in $\F_{q^n}$.
%\end{proposition}
%\begin{proof}
%Let $f \in \F_q[x]$ be a factor of $x^n-1$ of degree $k$.
%We may use Proposition \ref{sieve-prop} with $\ell=q^n-1$ and $g$ a divisor of $x^n-1$ such that
%$\gcd(g,\frac{x^n-1}{f})=1$ and any irreducible factor of $x^n-1$ divides $g$ or
%$\frac{x^n-1}{f}$.
%
%Let $P_1,\ldots , P_s$ be all the  irreducible polynomials such that
%$\mathrm{rad}(\frac{x^n-1}{f})= P_1 \cdot P_2 \dots P_s$.
%Then
%$\delta=1-\sum_{i=1}^{s}\frac{1}{q^{\deg P_i}} \geq 1 - \frac{n-k}{q}
%\geq 1 - \frac{1}{n-k} = \frac{n-k-1}{n-k} >0$, since
%$q \geq (n-k)^2$ and $s \leq n-k$.  We also have that
%$$
%\Delta=2+\frac{s-1}{\delta}\leq \frac{n-k-1}{\frac{n-k-1}{n-k}}+2 = n-k+2.
%$$
%This means that $W(\ell) W(\widetilde{g})  \Delta \leq (n-k+2)W(q^n-1)$ and from 
%Proposition \ref{sieve-prop} we get
%	the desired result.
%\end{proof}

The next result is restricted to the case $n=2$ and improves the bound given by Theorem \ref{principal}. 

\begin{proposition}\label{casen2}
For any prime power $q$, if $\alpha \in \mathbb F_{q^2}$ is primitive, then $\alpha$ is normal
over $\mathbb F_{q}$.
\end{proposition}
\begin{proof}
Assume that $\alpha$ is not normal. Thus $\{\alpha, \alpha^q\}$ is linearly dependent over $\Fq$, which implies that $\frac{\alpha^q}{\alpha} \in \Fq$. Therefore, $\alpha^{(q-1)^2} = (\alpha^{q-1})^{q-1} = 1$. Since $(q-1)^2 < q^2-1$, it follows that $\alpha$ is not primitive, a contradiction.

\end{proof}

\begin{observation}\label{observation2}
From Proposition \ref{casen2} and \cite[Theorem 3]{Cohen2015}, it follows that if 
$q \ge (m-1)W(q^2-1)^m,$
then $(q,2) \in N_m$, and from \cite[Theorem 5]{Cohen2015}
if $q > (m-1)\left(\frac{mr-1}{\delta}+2 \right) W(e)^m,$
then $(q,2) \in N_m$, where $e$ is a divisor of $q^2-1$ and 
$\delta=1 - m \sum_{i=1}^r \frac{1}{p_i}>0$, where $p_1,\ldots , p_r$ 
(for $r \ge 0$) are the primes dividing $q^2-1$ but not $e$.

In fact,  \cite[Theorem 3]{Cohen2015} and  \cite[Theorem 5]{Cohen2015}
deal with $\beta =1$, but if we look at the whole paper, we
observe that they work for all $\beta \in \mathbb{F}_{q}^*.$
\end{observation}

\section{Asymptotic results}\label{sectionasymp}

To apply Theorem \ref{principal}, in order to obtain asymptotic results, we need (among other results) an upper bound of $W(u)$. The following result was inspired
by \cite[Lemma 2.6]{lenstra}.

\begin{proposition}\label{boundW}
	Let $r$ be a positive integer,
	$p_1, \ldots , p_r$ be the list of the first $r$ prime numbers, and $P_r = p_1 \cdot \ldots \cdot p_r$
	be its product.
	For every positive integer $u \geq P_r$ we have $W(u) \leq u^t$, where $t$ is a real number
	satisfying $t \geq \frac{r \log 2}{\log P_r}$.
\end{proposition}
\begin{proof}
	Let $u\geq P_r$ be a positive integer. If $\omega(u) \leq r,$ then
	\[
	W(u)=2^{\omega(u)}\leq 2^r \leq P_r^t \leq u^t.
	\]
	If $\omega(u)>r$, write $u=u_1 u_2$, where the prime factors of $u_1$ are the first
	$r$ prime numbers which divide $u$, and $u_2$ is a positive integer such that $\gcd(u_1,u_2)=1$.
	Then $W(u_1)\leq u_1^t$, since $u_1 \geq P_r$ and $\omega(u_1)=r$. Observe also that
	for every prime factor $\tilde{p}$ of $u_2$ we have $2 < \tilde{p}^t$, since
	$2^r \leq P_r^t <\tilde{p}^{rt}$. This implies that $W(u_2) = 2^{\omega(u)-r}\leq u_2^t$. Thus
	$W(u)=W(u_1)W(u_2)\leq u_1^t u_2^t = u^t$.

\end{proof}

\begin{corollary} \label{boundW10}
	Let $u$ be a positive integer. If $u \ge 7.51 \cdot 10^{358}$, then $W(u) \le u^{\frac{1}{8}}.$ If $u\geq 1.39 \cdot 10^{1424},$ then $W(u) \leq u^{\frac{1}{10}}.$  If $u \ge 3.31 \cdot 10^{2821}$, then $W(u) \le u^{\frac{1}{11}}.$
\end{corollary}
\begin{proof}
	If we consider $r=149$, we get $P_{149}<7.51 \cdot 10^{358}$ and $\frac{1}{8} > \frac{149 \log 2}{\log P_{149}}$.
	If we consider $r=473$, we get $P_{473}<1.39 \cdot 10^{1424}$ and $\frac{1}{10} > \frac{473 \log 2}{\log P_{473}}$.
	If we consider $r=852$, we get $P_{852}<3.31 \cdot 10^{2821}$ and $\frac{1}{11} > \frac{852 \log 2}{\log P_{852}}$. The result follows directly from Proposition \ref{boundW}.
\end{proof}

%From Corollary \ref{boundW10} we get a first asymptotic result for $m=3$. Before let us recall some results.
From Corollary \ref{boundW10} we get the first asymptotic results for $m=2$ and for $m=3$. Before let us recall some results.

\begin{lemma}\label{lem:3.7}
Let $q$ be a prime power, and let $n$ be a positive integer. The number of monic
irreducible factors of
$x^n-1$ over $\mathbb{F}_q$ is at most %less or equal than 
$\frac{n}{a}+b,$ where
the pair $(a,b)$ can be chosen among the following pairs:
\begin{align*}
&(1,0), \quad \left(2,\frac{q-1}2\right), \quad \left(3, \frac{q^2+3q-4}{6}\right), \\ 
&\left(4,\frac{q^3+3q^2+5q-9}{12}\right), \quad \left(5,\frac{3q^4+8q^3+15q^2+22q-48}{60}\right).
\end{align*}
\end{lemma}	
\begin{proof}
Let $s_{n,t}$ be the number of distinct monic irreducible polynomials of degree at most $t$ that divide 
$x^n-1,$ and let $T_{n,t}$ be the sum of their degrees.
From \cite[Lemma 3.7]{AN}, we have $W(x^n-1) = 2^j,$ where
\[
j \le \frac{n+(t+1)s_{n,t}-T_{n,t}}{t+1},
\]
and the right-hand side of the expression above maximizes when $s_{n,t}$ is maximal.
Since the number of monic polynomials of degree $i$ which divide $x^n-1$ is at most %less or equal than 
the numbers of elements of $\mathbb{F}_{q^i}$, which are not in  $\mathbb{F}_{q^j}$,
for every divisor $j$ of $i$, divided by $i$, we have that
for $t=0$ we may choose $a=1$ and $b=0$,
for $t=1$ the maximum value of
$s_{n,t}$ is $q-1$ and in this case $T_{n,t}=q-1$ (so $a=2$ and $b=\frac{q-1}{2}$), for $t=2$
the maximum value of
$s_{n,t}$ is $q-1+ \frac{q^2-q}{2}=\frac{q^2+q-2}{2}$ and in this case $T_{n,t}=q-1+q^2-q=q^2-1$
(so $a=3$ and $b=\frac{q^2+3q-4}{6}$), for $t=3$
the maximum value of
$s_{n,t}$ is $q-1+ \frac{q^2-q}{2}+\frac{q^3-q}{3}=\frac{2q^3+3q^2+q-6}{6}$ and in this case
$T_{n,t}=q-1+q^2-q+q^3-q=q^3+q^2-q-1$
(so $a=4$ and $b=\frac{q^3+3q^2+5q-9}{12}$), and for $t=4$
the maximum value of
$s_{n,t}$ is $q-1+ \frac{q^2-q}{2}+\frac{q^3-q}{3}+\frac{q^4-q^2}{4}=\frac{3q^4+4 q^3+3q^2+2q-12}{12}$ and in this case
$T_{n,t}=q-1+q^2-q+q^3-q+q^4-q^2=q^4+q^3-q-1$
(so $a=5$ and $b=\frac{3q^4+8q^3+15q^2+22q-48}{60}$).

\end{proof}

From the previous results, we obtain weaker versions of Theorem \ref{Teorema1}(i),(ii).

\begin{proposition}\label{bound2-1}
Let $q$ be a prime power, and $n \ge 2$ be an integer. If $q^n \geq 7.51\cdot 10^{358},$ then
$(q,n)\in N_2$.
\end{proposition}
\begin{proof}
From the previous lemma, there exist non-negative integers $a,b$
depending on $q$ such that
$W(x^n-1) \leq 2^{\frac{n}{a}+b}$. %Since $\frac{1}{8} \ge \frac{149 \log 2}{\log P_{149}}$,
From Corollary \ref{boundW10},
%we get that $W(M)\le M^{\frac{1}{8}}$ for $M \ge 7.51\cdot 10^{358} > P_{149}$. So, 
we have
$2W(x^n-1)W(q^n-1)^2 \leq 2 \cdot 2^{\frac{n}{a}+b} \cdot q^{\frac{n}{4}}$. 
This proposition follows provided $q^{\frac{n}{2}} \geq 2 \cdot 2^{\frac{n}{a}+b} \cdot q^{\frac{n}{4}}$,
which is equivalent to
$\left(\frac{q}{2^{\frac 4a}}\right)^n \geq 2^{4b+4}$.
Let $c$ be a positive integer, and suppose that $q \geq c = 2^{\log_2 c}$. 
If $1-\frac{4}{a} \log_c 2>0$, we get that
\[
\frac{q}{2^{\frac 4a}} \geq q^{1-\frac{4}{a} \log_c 2}.
\]
%	{\color{red}Let $c$ be a positive integer and suppose that $q \geq c$. There exist non-negative numbers $a,b$
%	depending on $q$ such that
%	$W(x^n-1) \leq 2^{\frac{n}{a}+b}$.}
%	From Lemma \ref{boundW} we have
%	$3W(x^n-1)W(q^n-1)^3 \leq 3 \cdot 2^{\frac{n}{a}+b} \cdot q^{\frac{3n}{10}}$, and
%	$q^{\frac{n}{2}} \geq 3 \cdot 2^{\frac{n}{a}+b} \cdot q^{\frac{3n}{10}}$ if and only if
%	$\left(\frac{q}{2^{5/a}}\right)^n \geq 3^5\cdot 2^{5b}$.
%	As $q \geq c = 2^{\log_2 c}$, 
%	if $1-\frac{5}{a} \log_c 2>0$, we get that
%	$$
%	\frac{q}{2^{5/a}} \geq q^{1-\frac{5}{a} \log_c 2}.
%	$$
So, let us verify if 
\begin{equation}\label{ineqm20}
\left(q^n \right)^{1-\frac{4}{a} \log_c 2} \geq 2^{4b+4}
\end{equation}
holds, where $1-\frac{4}{a} \log_c 2>0$. If $c=17$, $a=1$ and $b=0$, Inequality \eqref{ineqm20} holds, since 
$q^n \geq 7.51\cdot 10^{358}$.
For $5 \le 	q \le 16$
Inequality \eqref{ineqm20} holds if we choose $c=q$, $a=2$ and $b=\frac{q-1}{2}$,
according to Lemma \ref{lem:3.7}. Similarly, Inequality \eqref{ineqm20} holds for 
$(q,a,b,c) \in \{(2,5,\frac{14}{5},2) \} \cup \{(q,3,\frac{q^2+3q-4}{6},q) \mid q \in \{ 3,4\} \}.$
We conclude the proof from Theorem \ref{principal}.
\end{proof}

\begin{proposition}\label{bound3-1}
	Let $q$ be an odd prime power and $n \ge 2$ be an integer. If $q^n \geq 1.39 \cdot 10^{1424},$ then
	$(q,n)\in N_3$.
\end{proposition}
\begin{proof}
	From Lemma \ref{lem:3.7}, there exist non-negative integers $a,b$
	depending on $q$ such that
	$W(x^n-1) \leq 2^{\frac{n}{a}+b}$.
	From Corollary \ref{boundW10}, we have
	$3W(x^n-1)W(q^n-1)^3 \leq 3 \cdot 2^{\frac{n}{a}+b} \cdot q^{\frac{3n}{10}}$. This proposition follows provided $q^{\frac{n}{2}} \geq 3 \cdot 2^{\frac{n}{a}+b} \cdot q^{\frac{3n}{10}}$, which is equivalent to
	$\left(\frac{q}{2^{\frac 5a}}\right)^n \geq 3^5\cdot 2^{5b}$.
	Let $c$ be a positive integer, and suppose that $q \geq c = 2^{\log_2 c}$. 
	If $1-\frac{5}{a} \log_c 2>0$, we get that
	\[
	\frac{q}{2^{\frac 5a}} \geq q^{1-\frac{5}{a} \log_c 2}.
	\]
%	{\color{red}Let $c$ be a positive integer and suppose that $q \geq c$. There exist non-negative numbers $a,b$
%	depending on $q$ such that
%	$W(x^n-1) \leq 2^{\frac{n}{a}+b}$.}
%	From Lemma \ref{boundW} we have
%	$3W(x^n-1)W(q^n-1)^3 \leq 3 \cdot 2^{\frac{n}{a}+b} \cdot q^{\frac{3n}{10}}$, and
%	$q^{\frac{n}{2}} \geq 3 \cdot 2^{\frac{n}{a}+b} \cdot q^{\frac{3n}{10}}$ if and only if
%	$\left(\frac{q}{2^{5/a}}\right)^n \geq 3^5\cdot 2^{5b}$.
%	As $q \geq c = 2^{\log_2 c}$, 
%	if $1-\frac{5}{a} \log_c 2>0$, we get that
%	$$
%	\frac{q}{2^{5/a}} \geq q^{1-\frac{5}{a} \log_c 2}.
%	$$
	So, let us verify if 
	\begin{equation}\label{ineq0}
		\left(q^n \right)^{1-\frac{5}{a} \log_c 2} \geq 3^5 \cdot 2^{5b}
	\end{equation}
	holds, where $1-\frac{5}{a} \log_c 2>0$. If $c=37$, $a=1$ and $b=0$, Inequality \eqref{ineq0} holds, since 
	$q^n \geq 1.39 \cdot 10^{1424}$. Let $c=7$, and suppose that $q\leq 31$.
	As $q^n \geq 1.39 \cdot 10^{1424}$, 
	we have $n \geq \frac{\log(1.39 \cdot 10^{1424})}{\log 31}> 954$. So,
	$\gcd(n,q-1)\le q-1 \le \frac{30}{954}n <\frac{1}{30}n,$ and
	from \cite[Equation (2.10)]{lenstra} we get $W(x^n-1)\leq 2^{\frac{31n}{60}}$, which means that we may choose
	$a=\frac{60}{31}$ and $b=0$. In this case Inequality \eqref{ineq0} also holds.
	For $q=5$, from Lemma \ref{lem:3.7} we may choose $c=5$, $a=3$ and $b=6$, and
	Inequality \eqref{ineq0} holds. 
	Finally, for $q=3$ from
	Lemma \ref{lem:3.7} we may choose $c=3$, $a=4$ and $b=5$, and
	Inequality \eqref{ineq0} holds.	
	We conclude the proof from Theorem \ref{principal}.

\end{proof}

We may also use Proposition \ref{boundW} in order to get  asymptotic result for $m=4$, which proves Theorem \ref{Teorema1}(iii).

\begin{proposition}\label{bound4-1}
	Let $q$ be an odd prime power, and $n \ge 2$ be an integer. If $q^n \geq 3.31 \cdot 10^{2821},$ then
	$(q,n)\in N_4$.
\end{proposition}
\begin{proof}
	From Lemma \ref{lem:3.7}, there exist non-negative integers $a,b$
	depending on $q$, such that
	$W(x^n-1) \leq 2^{\frac{n}{a}+b}$.
	From Proposition \ref{boundW}, we have
	$4W(x^n-1)W(q^n-1)^4 \leq 4 \cdot 2^{\frac{n}{a}+b} \cdot q^{\frac{4n}{11}}$. This proposition follows provided $q^{\frac{n}{2}} \geq 4 \cdot 2^{\frac{n}{a}+b} \cdot q^{\frac{4n}{11}}$, which is equivalent to
	$\left(\frac{q}{2^{\frac{22}{3a}}}\right)^n \geq 2^{\frac{22(b+2)}{3}}$.
	Let $c$ be a positive integer, and suppose that $q \geq c = 2^{\log_2 c}$. If $1-\frac{22}{3a} \log_c 2>0$, we get that
	\[
	\frac{q}{2^{\frac{22}{3a}}} \geq q^{1-\frac{22}{3a} \log_c 2}.
	\]
	So, let us verify if 
	\begin{equation}\label{ineq}
		\left(q^n \right)^{1-\frac{22}{3a} \log_c 2} \geq2^{\frac{22(b+2)}{3}}
	\end{equation}
	holds, where $1-\frac{22}{3a} \log_c 2>0$. If $c=163$, $a=1$ and $b=0$, Inequality \eqref{ineq} holds, since 
	$q^n \geq 3.31 \cdot 10^{2821}$. Let $c=19$ and suppose that $q\leq 157$.
	As $q^n \geq 3.31 \cdot 10^{2821}$, we have $n \geq \frac{3.31 \cdot 10^{2821}}{\log 157}>1284$. So,
	$\gcd(n,q-1)\le q-1 \le \frac{156}{1285}n <\frac{1}{8}n,$ and
	from \cite[Equation (2.10)]{lenstra} we get $W(x^n-1)\leq 2^{\frac{9n}{16}}$, which means that we may choose
	$a=\frac{16}{9}$ and $b=0$. In this case, Inequality \eqref{ineq} also holds.
	For $q = 17$, consider $c=17$ and from Lemma \ref{lem:3.7}, we may choose $a=2$ and $b=8$.
	In this case, Inequality \eqref{ineq} holds.
	From Lemma \ref{lem:3.7}, we may choose $a=3$ and $b=\frac{q^2+3q-4}{6}$. For
	$c=q$ and $q\in \{7,9,11,13 \},$ Inequality \eqref{ineq} holds. For $q=5$, again from Lemma \ref{lem:3.7} we may choose $c=5$, $a=4$ and $b=18$, and
	Inequality \eqref{ineq} holds.
	Finally, for $q=3$, from Lemma \ref{lem:3.7}, we may choose $c=3$, $a=5$ and $b=\frac{51}{5}$, and
	Inequality \eqref{ineq} holds.
	We conclude the proof from Theorem \ref{principal}.

\end{proof}

%{\color{red} Ideia: no lugar dessa Proposi\c{c}\~ao 4.5, poder\'iamos fazer algo geral, como fixado $m \in \mathbb N$, existem $c(m)$ e $q_0(m)$ comput\'aveis tais que se $q \ge c(m)$ e $q^n > q_0(m)$ ent\~ao $(q,n) \in N_m$.}

%\begin{proposition}\label{bound5-1}
%Let $q$ be an odd prime power and $n$ a positive integer. If $q^n \geq 3.31 \cdot 10^{2821}$ {\color{red}(n\~ao seria $3.55 \cdot 10^{172,803}$?)}
%and $q>4.56\cdot 10^6$ {\color{red}(n\~ao seria $4194304 \sim 4.20 \cdot 10^6$?)}, then
%$(q,n)\in N_5$.
%\end{proposition}
%\begin{proof}
%From the proof of Proposition \ref{bound4-1}, for $u \geq 3.31 \cdot 10^{2821}$
%we have $W(u) \leq u^{\frac{1}{11}}$.
%Suppose that $q >{\color{red}c=4.56\cdot 10^6 (4.2 \cdot 10^6?)}$.
%From Lemma \ref{boundW} we have
%$5W(x^n-1)W(q^n-1)^5 \leq 5 \cdot 2^{n} \cdot q^{\frac{5n}{11}}$, and
%$q^{\frac{n}{2}} \geq 5 \cdot 2^{n} \cdot q^{\frac{5n}{11}}$ if and only if
%$\left(\frac{q}{2^{22}}\right)^n \geq 5^{22}$.
%As $q \geq c = 2^{\log_2 c}$ and
%$1- 22\log_c 2 >0$, we have that
%$$
%\frac{q}{2^{22}} \geq q^{1- 22\log_c 2}.
%$$
%Now, we conclude from Theorem \ref{principal}, since
%$$
%\left(q^n \right)^{1- 22\log_c 2} \geq 5^{22}.
%$$
%\end{proof}
	
For $m > 4$ we have the following result, which proves Theorem \ref{Teorema1}(iv).
\begin{proposition}\label{boundm}
Let $q$ be an odd prime power, and $n,m$ be positive integers with $m \ge 5$. 
There exist positive constants $c(m)$ and $q_0(m)$ depending on $m$ such that
if $q \ge c(m)$ and $q^n \ge q_0(m),$ then $(q,n) \in N_m$.
\end{proposition}
\begin{proof}
From Proposition \ref{boundW}, there exists a contant $q_0(m)$ such that for every $u\ge q_0(m)$
we have
$W(u)\le u^{\frac{1}{2(m+1)}}$, and in this case we have
$mW(x^n-1)W(q^n-1)^m \leq m \cdot 2^{n} \cdot q^{\frac{mn}{2(m+1)}}$.
Since 
$q^{\frac{n}{2}} \geq m \cdot 2^{n} \cdot q^{\frac{mn}{2(m+1)}}$ is equivalent to
$\left(\frac{q}{2^{2(m+1)}}\right)^n \geq m^{2(m+1)}$, 
from Theorem \ref{principal}, if the last inequality holds, then
$(q,n)\in N_m$. We will prove that
if $q \ge c(m)=2^{4(m+1)},$ then $\left(\frac{q}{2^{2(m+1)}}\right)^n \geq m^{2(m+1)}$.

From Proposition \ref{boundW}, we may choose $q_0(m)=P_r,$ where $r$ is a positive integer such that
$\frac{1}{2(m+1)} \geq \frac{r \log 2}{\log P_r}$, thus $P_r \ge 2^{2r(m+1)}$. This implies
that the $r$-th prime $p_r$ satisfies $p_r \ge 2^{2(m+1)}>m^2$. Let $\pi(x)$ be
the number of prime numbers %less or equal than 
up to $x$. From \cite[Theorem 4.6]{Apostol}, 
we have $\pi(2^r)> \frac{2^r}{6\log 2^r} >r$ since $r>8$ (in fact, from Proposition \ref{bound4-1},
for $m \ge 5$ we have $r \ge 852$).
In particular $2^r>p_r>m^2$. Putting all this together, for $q \ge c(m)$, we get that 
\[
\left(\frac{q}{2^{2(m+1)}}\right)^n \ge q^{\frac n2} \ge  P_r^{\frac12} \ge 2^{r(m+1)}>m^{2(m+1)}.
\]
\end{proof}

\section{The case $m=3$}\label{sec:m=3}

From now on we concentrate on the case $m=3$.
We will use the sieving technique in order to decrease the bound of Proposition \ref{bound3-1}.
Before, we need some notations. For $k,r \in \mathbb N$, let $\omega(k)$ be the number of primes dividing $k$, $P(k,r)$ be the product of the first $r$ primes greater than or equal to $k,$ and $S(k,r)$ be the sum of the inverses of these primes.

\begin{proposition}\label{bound3-2}
Let $q$ be an odd prime power, and $n \ge 2$ be an integer. If $q^n \geq 3.422\cdot 10^{40}$
or if $q \ge 37$ and $ q^n \ge 7.391\cdot 10^{38},$
then
$(q,n)\in N_3.$
\end{proposition}

\begin{proof}
We will use the notation of Proposition \ref{sieve-prop}. Suppose that
$q^n <  1.39 \cdot 10^{1424},$ since for $q^n \ge 1.39 \cdot 10^{1424}$ we already have
that $(q,n)\in N_3.$ Let $e$ be the product of the primes less than $353$  which divide $q^n-1,$ and
let $r$ be the number of primes greater or equal %than 
to $353$ which divide $q^n-1.$
%Let also $P(353,r)$ be the product of the first $r$ primes greater than $353$ and let $S(353,r)$ be the sum of the inverse of those primes. 
Since $P(353,r)<q^n-1<1.39 \cdot 10^{1424},$ then $r \le 442.$ If $u=\omega(e),$ %(definir antes como sendo o n�mero de primos que dividem e)
then $u \le 70,$ since there are $70$ primes less than $353.$ Choose
$r(u)=\max \{ r \mid P_u \cdot P(353,r) \le 1.39 \cdot 10^{1424} \},$ where
$P_u$ is the product of the first $u$ prime numbers. Let also
$\delta(353,r(u))= 1 - 3 S(353,r(u))$ and $\Delta(353,r(u))= 2+ \dfrac{3r(u)-1}{\delta(353,r(u))}.$
Let $\delta$ and $\Delta$ be as in the Proposition \ref{sieve-prop} with $g=x^n-1$.
Observe that
$r \le  r(u)\le r(0) = 442$ and $\delta \ge \delta (353,r(u)) \ge \delta(353,r(0)) >0,$
so we may apply Proposition \ref{sieve-prop}.

As in Proposition \ref{bound3-1}, let $c$ be a positive integer, and suppose that $q \geq c$. There exist non-negative integers $a,b$
depending on $q$ such that
$W(x^n-1) \leq 2^{\frac{n}{a}+b}$.
So, we have
$3\Delta W(x^n-1)W(e)^3 \leq 3 \cdot \Delta(353,r(u)) \cdot 2^{\frac{n}{a}+b} \cdot 2^{3u}.$
In order to apply Proposition \ref{sieve-prop} we need to find a lower bound of $q^n$. In that sense, we will study the inequality $q^{\frac n2} \ge 3 \cdot \Delta(353,r(u)) \cdot 2^{\frac{n}{a}+b+ 3u}$.

As $q \geq c = 2^{\log_2 c}$, if
\begin{equation}\label{ineqdelta}
3 \cdot \Delta(353,r(u)) \cdot 2^{3u+b} \cdot q^{\frac{n}{a}\log_c 2} \leq q^{\frac n2},
\end{equation}
then $q^{\frac n2} \ge 3\Delta W(x^n-1)W(e)^3.$ If $a>\log_c 4,$ then Inequality \eqref{ineqdelta} is equivalent to
\[
q^n \ge \left( 3 \cdot \Delta(353,r(u)) \cdot 2^{3u+b}\right)^{\frac{2a}{a - \log_c 4}}.
\]
From Lemma \ref{lem:3.7}, we may choose
$(a,b,c) \in A,$ where
\[
A=\{(1,0,37),(3,6,5),(4,5,3) \} \cup \left\{ \left(2,\frac{q-1}2,q\right) \mid 7 \le q \le 31  \right\}.
\]
We get
\[
\max \{ \left( 3 \cdot \Delta(353,r(u)) \cdot 2^{3u+b}\right)^{\frac{2a}{a - \log_c 4}} 
\mid (a,b,c)\in A \textrm{ and } 0 \le u \le 70
  \} < 2.129\cdot 10^{221},
\]
and from Proposition \ref{sieve-prop} we conclude that if $q^n \geq 2.129\cdot 10^{221},$ then $(q,n)\in N_3.$

Suppose now that
$q^n <  2.129\cdot 10^{221},$ and repeat the process with $e$
being the product of the primes less than $101$  which divide $q^n-1,$ and
with $r$ being the number of primes greater than or equal %than 
to $101$ which divide $q^n-1.$
In this case, for $u=\omega(e)$ we have $u \le 25.$ 
When $q \leq 31$, we observe that the characteristic of $\mathbb{F}_q$ is less than $101$ and does not divide $q^n-1,$
%which is less than 101, does not divide $q^n-1,
hence  we may suppose $u \le 24.$
We get
\[
\max \{ \left( 3 \cdot \Delta(101,r(u)) \cdot 2^{3u+b}\right)^{\frac{2a}{a - \log_c 4}} 
\mid (a,b,c)\in A \textrm{ and } 0 \le u \le 25
\} < 7.525\cdot 10^{85},
\]
and from Proposition \ref{sieve-prop} we get that if $q^n \geq 7.525\cdot 10^{85},$ then $(q,n)\in N_3.$

Repeat the process with $e$
being the product of the primes less than $53$  which divide $q^n-1$ and
with $r$ being the number of primes greater than or equal %than 
to $53$ which divide $q^n-1.$ 
We get
\[
\max \{ \left( 3 \cdot \Delta(53,r(u)) \cdot 2^{3u+b}\right)^{\frac{2a}{a - \log_c 4}} 
\mid (a,b,c)\in A \textrm{ and } 0 \le u \le 15
\} < 7.871\cdot 10^{54},
\]
where $0 \le u \le 14$ if $c<37.$

Repeat this process several times. At each step we choose a prime $\tilde{p}$ such that
$e$ is the product of the primes less than $\tilde{p}$ which divide $q^n-1,$ and
$r$ is the number of primes greater or equal %than 
to $\tilde{p}$ which divide $q^n-1.$ But in all
those cases 
the maximum value of  $\left( 3 \cdot \Delta(\tilde{p},r(u)) \cdot 2^{3u+b}\right)^{\frac{2a}{a - \log_c 4}}$
is calculated for $(a,b,c)\in \tilde{A},$ where
$
\tilde{A} = \{(1,0,37),(1,0,31),(3,6,5),(4,5,3) \} 
\cup \{ (2,\frac{q-1}{2},q) \mid 7 \le q \le 29  \}.$
Table \ref{sievetable} summarize the process for all odd prime powers.

{\small
\begin{table}[h]
	\centering
	\begin{tabular}{c|c|c}
$q^n <M$  & $\tilde{p}$ & Maximum value \\
\hline
$M=7.871\cdot 10^{54}$ & $41$  & $1.368\cdot 10^{45}$ \\ \hline
$M=1.368\cdot 10^{45}$ & $37$  & $7.379\cdot 10^{41}$ \\ \hline
$M=7.379\cdot 10^{41}$ & $31$  & $3.422\cdot 10^{40}$ \\ \hline
\end{tabular}\vspace*{0.5cm}
\caption{Sieving process for all odd prime power $q$.}
\label{sievetable}
\end{table}
}

Table \ref{sievetable37} shows the sieving process for $q \ge 37$.
{\small
\begin{table}[h]
	\centering
	\begin{tabular}{c|c|c}
		$q^n <M$  & $\tilde{p}$ & Maximum value \\
		\hline
		$M=3.422\cdot 10^{40}$ & $37$  & $1.71\cdot 10^{39}$ \\ \hline
		$M=1.71\cdot 10^{39}$ & $31$  & $7.391\cdot 10^{38}$ \\ \hline
	\end{tabular}\vspace*{0.5cm}
	\caption{Sieving process for $q \ge 37$.}
	\label{sievetable37}
\end{table}
}
\end{proof}

\subsection{Proof of Theorem \ref{Teorema1}(ii)}\label{sectionproof}

Observe that from Proposition \ref{bound3-2}, if 
$n \geq 85,$ then $(q,n)\in N_3$ for any odd prime power $q,$ since $q^n\ge 3^{85} > 3.592\cdot 10^{40}>3.422\cdot 10^{40}.$ 
In this way, from Proposition \ref{bound3-2},  we test Proposition \ref{sieve-prop} for every factorization of $q^n-1$ and every factorization of $x^n-1$ over $\mathbb{F}_q,$  for all integer number $n$
between $7$ and $84,$ and every odd prime power $q$ between $3$ and
\[
M_n=
\left\{
\begin{array}{cl}
\lceil \sqrt[n]{3.422\cdot 10^{40}} \rceil & \textrm{if }  \sqrt[n]{3.422\cdot 10^{40}} < 31;\\
\lceil \sqrt[n]{7.391\cdot 10^{38}} \rceil & \textrm{if }  \sqrt[n]{7.391\cdot 10^{38}} > 31;\\
31                           & \textrm{otherwise},
\end{array}
\right.
\]
where $\lceil x \rceil$ denotes the ceil of $x\in \mathbb{R}.$

We get that the condition $q^{\frac{n}{2}} \geq 3 \Delta W(g) W(e)^3$  holds for some values of $g$ and $e$, for every pair $(q,n)$ with $n \ge 7$ except for those displayed in Table \ref{m3nmaior6}.

{\small
\begin{table}[H]
\centering
\begin{tabular}{c|c}
$n$ & $q$ \\
\hline
$7$ & $3,7$ \\ \hline
$8$ & $3,5,7,9,11,13$ \\ \hline
$9$ & $3,7$ \\ \hline
\end{tabular}
\hspace{0.3cm}
\begin{tabular}{c|c}
$n$ & $q$ \\
\hline
$10$ & $3,5,11$ \\ \hline
$12$ & $3,5,7,13$ \\ \hline
$16$ & $3$ \\ \hline
\end{tabular} %\hspace{0.3cm}
\vspace*{0.5cm}
	\caption{Possible exceptions for Theorem \ref{Teorema1}(ii).}
	\label{m3nmaior6}
\end{table}
}
%Using SageMath, we get that the condition $q^{\frac{n}{2}} \geq 3 \Delta W(g) W(e)^3$ 
%holds for some values of $g$ and $e$, for every pair
%$(q,n)$ (with $n \ge 7$) except for 
%$(3, 7),$ 
%$(7, 7),$
%$(3, 8),$
%$(5, 8),$
%$(7, 8),$
%$(9, 8),$
%$(11, 8),$
%$(13, 8),$
%$(3, 9),$
%$(7, 9),$
%$(3, 10),$
%$(5, 10),$
%$(11, 10),$
%$(3, 12),$
%$(5, 12),$
%$(7, 12),$ 
%$(13, 12),$ 
%$(3, 16).$

In order to conclude Theorem \ref{Teorema1}(ii), we deal with each value of $n \in \{2,3,4,5,6\}$ separately.
%Let us prove some intermediate lemmas in order to finish the proof of Theorem \ref{Teorema1}.
We will need the following result.
\begin{lemma}{\cite[Lemma 4.1]{Kapetanakis-Reis}}\label{cota-t}
	Let $M$ be a positive integer, and $t$ be a positive real number.
	Then
	$W(M) \leq A_t \cdot M^{\frac{1}{t}},$
	where
\[A_t=\prod_{\substack{\wp < 2^t, \; \wp \mid M \atop \wp \text{ is prime}}}
	\frac{2}{\sqrt[t]{\wp}}.\]
\end{lemma}

Let us start with the case $n=6$.

\begin{lemma}\label{n6m3}
Let $q$ be an odd prime power and $n =6$. Then $(q,6)\in N_3$ except possibly for $q \in \{3,5,7,9,11,13,17,19,23,25,29,31,37,43,61\}$.
%the pairs
%$(3, 6),$
%$(5, 6),$
%$(7, 6),$
%$(9, 6),$
%$(11, 6),$
%$(13, 6),$
%$(17, 6),$
%$(19, 6),$
%$(23, 6),$
%$(25, 6),$
%$(29, 6),$
%$(31, 6),$
%$(37, 6),$
%$(43, 6),$ and
%$(61, 6)$.
\end{lemma}
\begin{proof}
From Proposition \ref{bound3-2} we get that $(q,6)\in N_3$ for $q \ge 3006888.$ Suppose
now that $q$ is an odd prime power such that $q<M=3006888$. We will use 
Proposition \ref{sieve-prop}
with $g=1$, and either $e=(q^2-1)$ if $7 \nmid q^6-1$ or $e=7(q^2-1)$ if $7 \mid q^6-1$.
Let $\{p_1,\dots,p_r\}$ be the set of primes from Proposition \ref{sieve-prop}. For $i \in \{1,\dots,r\}$, we have $p_i \mid q^6-1$ and $p_i \nmid q^2-1$, so $3 \mid \varphi(p_i)=p_i-1$ and $p_i \neq 2$ since $2 \mid q^2-1$.
This means that the set $\{ p_1,\ldots , p_r\}$ comprises %is composed by 
primes of the form $6j+1$. Let $S_r$ and $P_r$ be respectively
the sum of the inverses and the product of the first $r$ primes of the form
$6j+1,$ counting from $13.$ Since $\{ p_1,\ldots , p_r\}$ is a set of primes which divide
$q^4+q^2+1$ and $7 \notin \{ p_1,\ldots , p_r\},$ then
$P_r \le \prod_{i=1}^r p_i \le q^4 + q^2+1<8.175 \cdot 10^{25}$, therefore
$r\le 14$ and $S_r < 0.3141$. If we suppose $q>10^4,$ then
\[
\delta = 1 - 3 \sum_{i=1}^r \frac{1}{p_i} - \sum_{i=1}^s \frac{1}{q^{\deg h_i}} \ge 1 - 3\cdot S_r - \frac{6}{q}
> 0.0571
\]
and $\Delta = 2 + \frac{3r+s-1}{\delta}<825.118$. Observe that if
$q\ge (3\cdot 825.118\cdot 2^3 \cdot A_t^3)^{\frac{t}{3t-6}}$ for some real number $t>2,$ then
$q^3\ge 3\cdot 825.118 \cdot (2\cdot A_t \cdot q^{\frac2t})^3 > 3 \Delta W(1)W(e)^3$,
since, from Lemma \ref{cota-t}, $W(e) \le 2 W(q^2-1)< 2\cdot A_t \cdot q^{\frac2t}$. Thus,
from Proposition
\ref{sieve-prop}, if 
$q\ge (3\cdot 825.118\cdot 2^3 \cdot A_t^3)^{\frac{t}{3t-6}}$ for some real number $t>2$,
we have $(q,6)\in N_3$. For $t=4.53$ the condition above becomes
$q \ge 13051$.

For smaller values, we can verify directly the condition $q^3 \geq 3 \Delta W(g) W(e)^3$. In fact, it holds for some values of $g$ and $e$, for every pair $(q,6)$, where $q$ is an odd prime power
between $3$ and $13050$, except for $q \in \{3,$ $5,$ $7,$ $9,$ $11,$ $13,$ $17,$ $19,$ $23,$ $25,$ $29,$ $31,$ $37,$ $43,$ $61\}.$
%$(3, 6),$
%$(5, 6),$
%$(7, 6),$
%$(9, 6),$
%$(11, 6),$
%$(13, 6),$
%$(17, 6),$
%$(19, 6),$
%$(23, 6),$
%$(25, 6),$
%$(29, 6),$
%$(31, 6),$
%$(37, 6),$
%$(43, 6),$ and
%$(61, 6)$.

\end{proof}

For $n=5$, we use a similar argument than case $n=6$.

\begin{lemma}\label{n5m3}
Let $q$ be an odd prime power and $n =5$. Then $(q,5)\in N_3$ except possibly for
$q \in \{3,5,7,9,11,13,19,31\}$.
\end{lemma}
\begin{proof}
From Proposition \ref{bound3-2} we get that $(q,5)\in N_3$ for $q \ge 59393736.$ Suppose
now that $q$ is an odd prime power such that $q<M=59393736$. We will use 
Proposition \ref{sieve-prop}
with $g=1$ and $e=q-1$. Let $\{p_1,\dots,p_r\}$ be the set of primes from Proposition \ref{sieve-prop}. For $i \in \{1,\dots,r\}$, we have $p_i \mid q^5-1$ and $p_i \nmid q-1$, so $5 \mid \varphi(p_i)=p_i-1$ and $p_i \neq 2$ since $2 \mid q-1$.
	This means that the set $\{ p_1,\ldots , p_r\}$ comprises %is composed by 
	primes of the form $10j+1$. Let $S_r$ and $P_r$ be respectively
	the sum of the inverses and the product of the first $r$ primes of the form
	$10j+1.$ Since $\{ p_1,\ldots , p_r\}$ is a set of primes which divide
	$q^4+q^3+q^2+q+1,$ then
	$P_r \le \prod_{i=1}^r p_i \le q^4+q^3+q^2+q+1<1.2445 \cdot 10^{31}$, therefore
	$r\le 15$ and $S_r < 0.2331$. If we suppose $q>435,$ then
	\[
	\delta = 1 - 3 \sum_{i=1}^r \frac{1}{p_i} - \sum_{i=1}^s \frac{1}{q^{\deg h_i}} \ge 1 - 3 \cdot S_r - \frac{5}{q}
	> 0.2892
	\]
	and $\Delta = 2 + \frac{3r+s-1}{\delta}<171.433$. Observe that if
	$q\ge (3\cdot 171.433\cdot A_t^3)^{\frac{2t}{5t-6}}$ for some real number $t>2,$ then
	$q^{\frac{5}{2}} \ge 3\cdot 171.433 \cdot (A_t \cdot q^{\frac1t})^3 > 3 \Delta W(1)W(e)^3$,
	since, from Lemma \ref{cota-t}, $W(e) \le W(q-1)< A_t \cdot q^{\frac1t}$. Thus,
	from Proposition
	\ref{sieve-prop}, if 
	$q\ge (3\cdot 171.433\cdot A_t^3)^{\frac{2t}{5t-6}}$ for some real number $t>2$,
	we have $(q,5)\in N_3$. ofr $t=3.4$ the condition above becomes
	$q \ge 439$.
	
For smaller values, we can verify directly the condition $q^{\frac 52} \geq 3 \Delta W(g) W(e)^3.$ In fact, it holds for some values of $g$ and $e$, for every pair $(q,5)$, where $q$ is an odd prime power
	between $3$ and $439$, except for $q \in \{3,$ $5,$ $7,$ $9,$ $11,$ $13,$ $19,$ $31\}$. %the pairs
%	$(3, 5),$ $(5, 5),$ $(7, 5),$ $(9, 5),$ $(11, 5),$ $(13, 5),$ $(19, 5),$ and $(31, 5)$.

\end{proof}

Now, for $n=4$ we obtain the following result.

\begin{lemma}\label{n4m3}
Let $q$ be an odd prime power and $n =4$. Then $(q,4)\in N_3$ except possibly for
$q \in$
$\{3,$ 
$5,$ 
$7,$ 
$9,$ 
$11,$ 
$13,$ 
$17,$ 
$19,$ 
$23,$ 
$25,$ 
$27,$ 
$29,$ 
$31,$ 
$37,$ 
$41,$ 
$43,$ 
$47,$ 
$49,$ 
$53,$ 
$59,$ 
$61,$ 
$67,$ 
$71,$ 
$73,$ 
$79,$ 
$83,$ 
$89,$ 
$97,$ 
$101,$ 
$103,$ 
$109,$ 
$113,$ 
$127,$ 
$131,$ 
$137,$ 
$139,$ 
$149,$ 
$151,$ 
$157,$ 
$167,$ 
$169,$ 
$173,$ 
$181,$ 
$191,$ 
$197,$ 
$211,$ 
$229,$ 
$239,$ 
$281,$ 
$307,$ 
$419,$ 
$421,$ 
$463,$ 
$659,$ 
$727\}.$
\end{lemma}
\begin{proof}
From Proposition \ref{bound3-2} we get that $(q,4)\in N_3$ for $q \ge 5214057313.$ Suppose now that $q$ is an odd prime power such that $q<M=5214057313$. We will use Proposition \ref{sieve-prop} with $g=1$ and either $e=(q^2-1)$ if $5 \nmid q^4-1$ or $e=5(q^2-1)$ if $5 \mid q^4-1$. Let $\{p_1,\dots,p_r\}$ be the set of primes from Proposition \ref{sieve-prop}.
% For $i \in \{1,\dots,r\}$, we have $p_i \mid q^4-1$ and $p_i \nmid q^2-1$, so $4 \mid \varphi(p_i)=p_i-1$ and $p_i \neq 2$ since $2 \mid q^2-1$.
As we have seen before, this means that the set $\{ p_1,\ldots , p_r\}$ comprises  
primes greater than $5$ of the form $4j+1,$ since
$p_i \mid q^4-1$ and $p_i \nmid q^2-1$ for all $i \in \{1,\dots,r\}.$
Let $S_r$ and $P_r$ be respectively the sum of the inverses and the product of the first $r$ primes of the form $4j+1,$ counting from $13.$ Since $\{ p_1,\ldots , p_r\}$ is a set of primes which divide
$q^2+1$ and $5 \notin \{ p_1,\ldots , p_r\},$ then
$P_r \le \prod_{i=1}^r p_i \le \frac{q^2+1}{2}<1.36 \cdot 10^{19}$, therefore
$r\le 11$. If we suppose $q>10^8,$ then
\[
\delta = 1 - 3 \sum_{i=1}^r \frac{1}{p_i} - \sum_{i=1}^s \frac{1}{q^{\deg h_i}} \ge 1 - 3\cdot S_r - \frac{4}{q}
> 0.0938
\]
and $\Delta = 2 + \frac{3r+s-1}{\delta}<385.796$. Observing that if
$q\ge (3\cdot 385.796 \cdot 2^3 \cdot A_t^3)^{\frac{t}{2t-6}}$ for some real number $t>3,$ then
$q^2\ge 3\cdot 385.796 \cdot (2\cdot A_t \cdot q^{\frac2t})^3 > 3 \Delta W(1)W(e)^3$,
since, from Lemma \ref{cota-t}, $W(e) \le 2 W(q^2-1)< 2\cdot A_t \cdot q^{\frac2t}$.
Notice that $8\mid q^2-1$, so we may replace $A_t$  from Lemma \ref{cota-t}
by
\[
\tilde{A}_t=\frac{2}{\sqrt[t]{8}} \prod_{\substack{\wp < 2^t \\ \wp \text{ is odd}\\ \text{prime}}}
\frac{2}{\sqrt[t]{\wp}}.
\]

Thus,
from Proposition
\ref{sieve-prop}, if 
$q\ge (3\cdot 385.796\cdot 2^3 \cdot \tilde{A}_t^3)^{\frac{t}{2t-6}}$ for some real number $t>3,$
we have $(q,4)\in N_3$. For $t=5.5$ the condition above becomes
$q \ge 1.74\cdot 10^8.$

%We repeat the same process for $q \le 1.74\cdot 10^8.$ 
%If we suppose $q \ge 10^7$
%we get $r\le 9,$ $\delta >0.15447$ and $\Delta < 196.213.$ As in the previous calculations,
%if $q\ge (3\cdot 196.213\cdot 2^3 \cdot \tilde{A}_t^3)^{\frac{t}{2t-6}}$ for some real number $t>3,$
%we have $(q,4)\in N_3$ and for $t=5.4$ the condition above becomes
%$q \ge 8.16\cdot 10^7.$

We will use again Proposition \ref{sieve-prop}
with the ideas from Proposition \ref{bound3-2}.
Let $e$ be the product of the primes less than $29$ which divide $q^4-1,$ and
let $r$ be the number of primes greater or equal to %than 
$29$ which divide $q^4-1.$
Since $16 \mid q^4-1$ we must choose $r$ such that
$P(29,r)<\frac{q^4-1}{16}<5.729 \cdot 10^{31}.$ Therefore $r \le 17.$ If $u=\omega(e),$
then $u \le 9,$ since there are $9$ primes less than $29.$ 
Recall that $P_u$ is the product of the first $u$ prime numbers.
Choose
$r(u)=\max \{ r \mid 8 \cdot P_u \cdot P(29,r) \le 5.729 \cdot 10^{31}\},$ 
since $8 \cdot P_u \cdot P(29,r) \le q^4-1.$ 
Let $\delta$ and $\Delta$ be as in the Proposition \ref{sieve-prop} with $g=1$,
and suppose that $q \ge 10^5.$ Let also
$\delta(29,r(u))= 1 - 3 S(29,r(u))-\frac{4}{10^5}$ and $\Delta(29,r(u))= 2+ \dfrac{3r(u)+4-1}{\delta(29,r(u))}.$
Observe that $3 \le s\le 4,$
$r \le  r(u)\le r(0)$ and $\delta \ge \delta (29,r(u)) \ge \delta(29,r(0)) >0,$
so we may apply Proposition \ref{sieve-prop}.

So, we have
$3\Delta W(g)W(e)^3 \leq 3 \cdot \Delta(29,r(u)) \cdot 2^{3u}.$
From Proposition \ref{sieve-prop}, if
$q^2 \ge 3 \cdot \Delta(29,r(u)) \cdot 2^{3u},$ then $(q,4)\in N_3.$
We get
\[
\max \{ \left( 3 \cdot \Delta(29,r(u)) \cdot 2^{3u}\right)^{\frac{1}{2}} 
\mid  0 \le u \le 9
\} < 300350,
\]
and  we conclude that if $q \geq 300350,$ 
then $(q,4)\in N_3.$

For smaller values, we can verify directly the condition $q^2 \geq 3 \Delta W(g) W(e)^3.$ In fact, it holds for some values of $g$ and $e$, for every pair $(q,4)$, where $q$ is an odd prime power
between $3$ and $300350$, except for the values of $q$ in the statement of this result.
%$q \in$
%$\{3,$ 
%$5,$ 
%$7,$ 
%$9,$ 
%$11,$ 
%$13,$ 
%$17,$ 
%$19,$ 
%$23,$ 
%$25,$ 
%$27,$ 
%$29,$ 
%$31,$ 
%$37,$ 
%$41,$ 
%$43,$ 
%$47,$ 
%$49,$ 
%$53,$ 
%$59,$ 
%$61,$ 
%$67,$ 
%$71,$ 
%$73,$ 
%$79,$ 
%$83,$ 
%$89,$ 
%$97,$ 
%$101,$ 
%$103,$ 
%$109,$ 
%$113,$ 
%$127,$ 
%$131,$ 
%$137,$ 
%$139,$ 
%$149,$ 
%$151,$ 
%$157,$ 
%$167,$ 
%$169,$ 
%$173,$ 
%$181,$ 
%$191,$ 
%$197,$ 
%$211,$ 
%$229,$ 
%$239,$ 
%$281,$ 
%$307,$ 
%$419,$ 
%$421,$ 
%$463,$ 
%$659,$ 
%$727\}.$

%$\{727,$ $659,$
%$463,$
%$421,$
%$419,$
%$307,$
%$281,$ $239,$ $229,$ $211,$ $197,$ $191,$ $181,$ $173,$ $13^2,$ $167,$ $157,$ $151,$ $149,$ $139,$ $137,$ $131,$ $127,$ $113,$ $109,$ $103,$ $101,$ $97,$ $89,$ $83,$ $79,$ $73,$ $71,$ $67,$ $61,$ $59,$ $53,$ $7^2,$ $47,$ $43,$ $41,$ $37,$ $31,$ $29,$ $3^3,$ $5^2,$ $23,$ $19,$ $17,$ $13,$ $11,$ $3^2,$ $7,$ $5,$ $3\}$.
\end{proof}

For $n=3$, we get a similar result.

\begin{lemma}\label{n3m3}
Let $q$ be an odd prime power and $n=3$. Then $(q,3)\in N_3$ except possibly for 
$q \in$
$\{3,$ 
$5,$ 
$7,$ 
$9,$ 
$11,$ 
$13,$ 
$17,$ 
$19,$ 
$23,$ 
$25,$ 
$27,$ 
$29,$ 
$31,$ 
$37,$ 
$41,$ 
$43,$ 
$47,$ 
$49,$ 
$53,$ 
$61,$ 
$67,$ 
$71,$ 
$73,$ 
$79,$ 
$81,$ 
$97,$ 
$103,$ 
$107,$ 
$109,$ 
$121,$ 
$127,$ 
$131,$ 
$139,$ 
$151,$ 
$157,$ 
$163,$ 
$169,$ 
$181,$ 
$191,$ 
$193,$ 
$211,$ 
$241,$ 
$277,$ 
$289,$ 
$331,$ 
$361,$ 
$373,$ 
$421,$ 
$463,$ 
$529,$ 
$541,$ 
$571,$ 
$631,$ 
$661,$ 
$691,$ 
$751,$ 
$841,$ 
$919,$ 
$961,$ 
$991,$ 
$1171,$ 
$1381,$ 
$4951\}.$ 

\end{lemma}
\begin{proof}
From Proposition \ref{bound3-2}, we get that $(q,3)\in N_3$ for $q \ge 9.042 \cdot 10^{12}.$ Suppose now that $q$ is an odd prime power such that $q<M=9.042 \cdot 10^{12}$. We will use Proposition \ref{sieve-prop} with $g=1$, and either $e=q-1$ if $7 \nmid q^3-1$ or $e=7(q-1)$ if $7 \mid q^3-1$. Let $\{p_1,\dots,p_r\}$ be the set of primes from Proposition \ref{sieve-prop}.
As we have seen before, this means that the set $\{ p_1,\ldots , p_r\}$ comprises  
primes greater than $7$ of the form $6j+1,$ since
$p_i \mid q^3-1$ and $p_i \nmid q-1$ for all $i \in \{1,\dots,r\}.$
Let $S_r$ and $P_r$ be respectively the sum of the inverses and the product of the first $r$ primes of the form $6j+1,$ counting from $13.$ Since $\{ p_1,\ldots , p_r\}$ is a set of primes which divide
$q^2+q+1$ and $7 \notin \{ p_1,\ldots , p_r\},$ then
$P_r \le \prod_{i=1}^r p_i \le q^2+q+1 < 8.1758 \cdot 10^{25}$, therefore
$r\le 14$. If we suppose $q>10^6,$ then
\[
\delta = 1 - 3 \sum_{i=1}^r \frac{1}{p_i} - \sum_{i=1}^s \frac{1}{q^{\deg h_i}} \ge 1 - 3\cdot S_r - \frac{3}{q}
> 0.0579
\]
and $\Delta = 2 + \frac{3r+s-1}{\delta}<761.931$. Observe that if
$q \ge (3\cdot 761.931 \cdot 2^3 \cdot A_t^3)^{\frac{2t}{3t-6}}$ for some real number $t>2,$ then
$q^{\frac32} \ge 3\cdot 761.931 \cdot (2\cdot A_t \cdot q^{\frac{1}{t}})^3 > 3 \Delta W(1)W(e)^3$,
since, from Lemma \ref{cota-t}, $W(e) \le 2 W(q-1)< 2\cdot A_t \cdot q^{\frac1t}$.

Thus, from Proposition \ref{sieve-prop}, if 
$q\ge (3\cdot 761.931\cdot 2^3 \cdot A_t^3)^{\frac{2t}{3t-6}}$ for some real number $t>2,$
we have $(q,3)\in N_3$. For $t=4.6$ the condition above becomes
$q \ge 1.5555\cdot 10^8.$

We repeat the same process for $q \le 1.5555\cdot 10^8.$ 
If we suppose $q \ge 10^6$ we get $r\le 9,$ $\delta > 0.19068$ and $\Delta < 154.0873.$ As in the previous calculations,
if $q\ge (3\cdot 154.0873\cdot 2^3 \cdot A_t^3)^{\frac{2t}{3t-6}}$ for some real number $t>2,$
we have $(q,3)\in N_3$. For $t=4.5$ the condition above becomes $q \ge 2.301\cdot 10^7.$

We will use again Proposition \ref{sieve-prop}
with the ideas from Proposition \ref{bound3-2}.
Let $e$ be the product of the primes less than $23$ which divide $q^3-1$ and
let $r$ be the number of primes greater or equal to %than 
$23$ which divide $q^3-1.$
We must choose $r$ such that
$P(23,r)<q^3-1<1.219 \cdot 10^{22}.$ Therefore $r \le 13.$ If $u=\omega(e),$
then $u \le 8,$ since there are $8$ primes less than $23.$ 
Recall that $P_u$ is the product of the first $u$ prime numbers.
Choose
$r(u)=\max \{ r \mid P_u \cdot P(23,r) \le 1.219 \cdot 10^{22}\}.$ 
Let $\delta$ and $\Delta$ be as in the Proposition \ref{sieve-prop} with $g=1$,
and suppose that $q \ge 10^5.$ Let also
$\delta(23,r(u))= 1 - 3 S(23,r(u))-\frac{3}{10^5}$ and $\Delta(23,r(u))= 2+ \dfrac{3r(u)+3-1}{\delta(23,r(u))}.$
Observe that $2 \le s\le 3,$
$r \le r(u)\le r(0)$ and $\delta \ge \delta (23,r(u)) \ge \delta(23,r(0)) >0,$
so we may apply Proposition \ref{sieve-prop}.

So, we have
$3\Delta W(g)W(e)^3 \leq 3 \cdot \Delta(23,r(u)) \cdot 2^{3u}.$
From Proposition \ref{sieve-prop}, if
$q^{\frac32} \ge 3 \cdot \Delta(23,r(u)) \cdot 2^{3u},$ then $(q,3)\in N_3.$
We get
\[
\max \{ \left( 3 \cdot \Delta(23,r(u)) \cdot 2^{3u}\right)^{\frac{2}{3}} \mid  0 \le u \le 8 \} < 3.0884 \cdot 10^6,\]
and we conclude that if $q \geq 3.0884 \cdot 10^6,$ 
then $(q,3)\in N_3.$

Now, suppose that $q < 3.0884 \cdot 10^6$. Again, let $e$ be the product of the primes less than $19$ which divides $q^3-1$ and let $r$ be the number of primes greater or equal to %than 
$19$ which divides $q^3-1.$ We choose $r$ such that $P(19,r)<q^3-1<2.946 \cdot 10^{19}.$ The same argument as before implies that $(q,3) \in N_3$ for $q \ge 821257.$

For smaller values, we can verify directly the condition $q^{\frac32} \geq 3 \Delta W(g) W(e)^3.$ In fact, it holds for some values of $g$ and $e,$ for every pair $(q,3),$ where $q$ is an odd prime power between $3$ and $821257,$ except for the values of $q$ in the statement of this result.

%Using SageMath, we get that the condition $q^{\frac32} \geq 3 \Delta W(e)^3$ holds for some value of $e$ (with $g=1$), for every pair $(q,3)$, where $q$ is an odd prime power between $3$ and $821257$, except for $9858$ values of $q$ (the greatest is $q=821101$).
%
%Using SageMath again, the condition $q^{\frac32} \geq 3 \Delta W(g) W(e)^3$ holds for some values of $g$ and $e$, for every pair $(q,3)$, where $q$ is an odd prime power between $3$ and $821101$, except for the values of $q$ in the statement of this result. 
%$q \in$
%$\{3,$ 
%$5,$ 
%$7,$ 
%$9,$ 
%$11,$ 
%$13,$ 
%$17,$ 
%$19,$ 
%$23,$ 
%$25,$ 
%$27,$ 
%$29,$ 
%$31,$ 
%$37,$ 
%$41,$ 
%$43,$ 
%$47,$ 
%$49,$ 
%$53,$ 
%$61,$ 
%$67,$ 
%$71,$ 
%$73,$ 
%$79,$ 
%$81,$ 
%$97,$ 
%$103,$ 
%$107,$ 
%$109,$ 
%$121,$ 
%$127,$ 
%$131,$ 
%$139,$ 
%$151,$ 
%$157,$ 
%$163,$ 
%$169,$ 
%$181,$ 
%$191,$ 
%$193,$ 
%$211,$ 
%$241,$ 
%$277,$ 
%$289,$ 
%$331,$ 
%$361,$ 
%$373,$ 
%$421,$ 
%$463,$ 
%$529,$ 
%$541,$ 
%$571,$ 
%$631,$ 
%$661,$ 
%$691,$ 
%$751,$ 
%$841,$ 
%$919,$ 
%$961,$ 
%$991,$ 
%$1171,$ 
%$1381,$ 
%$4951\}.$ 

 \end{proof}

Finally, for $n=2$, we invoke Proposition \ref{casen2} and Observation \ref{observation2} to obtain the following result, which
 concludes the proof of Theorem \ref{Teorema1}(ii).
\begin{lemma}\label{n2m3}
Let $q$ be an odd prime power and $n=2$. Then $(q,2)\in N_3$ except possibly for $1373$ values of $q$. %(the greatest is $q = 3847271$). 
See Appendix \ref{appendixm3n2} for the complete list.
\end{lemma}
\begin{proof}
From Proposition \ref{bound3-2}, we get that $(q,2)\in N_3$ for $q \ge 2.7187 \cdot 10^{19}.$ Suppose now that $q$ is an odd prime power such that $q<M=2.7187 \cdot 10^{19}$. We will use Observation \ref{observation2} with $g=1$, and either $e=q-1$ if $5 \nmid q+1$ or $e=5(q-1)$ if $5 \mid q+1$. Let $\{p_1,\dots,p_r\}$ be the set of primes from Observation \ref{observation2}.
As we have seen before, this means that the set $\{ p_1,\ldots , p_r\}$ comprises  
primes greater than $5$ of the form $4j+1,$ since
$p_i \mid q^2-1$ and $p_i \nmid q-1$ for all $i \in \{1,\dots,r\}.$
Let $S_r$ and $P_r$ be respectively the sum of the inverses and the product of the first $r$ primes of the form $4j+1,$ counting from $13.$ Since $\{ p_1,\ldots , p_r\}$ is a set of primes which divides $q+1$ and $5 \notin \{ p_1,\ldots , p_r\},$ then
$P_r \le \prod_{i=1}^r p_i \le q+1 < 2.7187 \cdot 10^{19}$, therefore
$r\le 11$. Let 
\[
\delta = 1 - 3 \sum_{i=1}^r \frac{1}{p_i} \ge 1 - 3\cdot S_r > 0.0938
\]
and $\Delta = 2 + \frac{3r-1}{\delta}<343.1514$. Observe that if
$q \ge (2^4 \cdot 343.1514 \cdot A_t^3)^{\frac{t}{t-3}}$ for some real number $t>3,$ then
$q \ge 2\cdot 343.1514 \cdot (2\cdot A_t \cdot q^{\frac1t})^3 > 2 \Delta W(e)^3$,
since, from Lemma \ref{cota-t}, $W(e) \le 2 W(q-1)< 2\cdot A_t \cdot q^{\frac1t}$.
For $t=5.5,$ the condition above becomes $q \ge 5.0381\cdot 10^{16}.$

%We repeat the same process for $q \le 5.0381\cdot 10^{16},$ obtaining $r\le 10,$ $\delta > 0.1235444$ and $\Delta < 236.733424.$ By choosing $t=5.5$, if $q \ge 2.22622\cdot 10^{16}$, then $q \ge (2\cdot 236.733424\cdot 2^3 \cdot A_t^3)^{\frac{t}{t-3}},$ thus we have $(q,2)\in N_3.$

We will use again Observation \ref{observation2}
with the ideas from Proposition \ref{bound3-2}.
Let $e$ be the product of the primes less than $29$ which divide $q^2-1$
(including its powers)
and
let $r$ be the number of primes greater or equal to %than 
$29$ which divide $q^2-1.$
We must choose $r$ such that
$P(29,r)<q^2-1<2.53825\cdot 10^{33},$
since we may suppose $q< 5.0381\cdot 10^{16}.$ Therefore $r \le 18.$ If $u=\omega(e),$
then $u \le 9,$ since there are $9$ primes less than $29.$ We also may suppose that $8 \mid e$ and $u \ge 1,$ since $8 \mid q^2-1$. 
Recall that $P_u$ is the product of the first $u$ prime numbers, so $4 P_u \le e.$
Choose
$r(u)=\max \{ r \mid 4\cdot P_u \cdot P(29,r) \le 2.53825\cdot 10^{33}\}.$ 
Let $\delta$ and $\Delta$ be as in the Observation \ref{observation2}.
Let also
$\delta(29,r(u))= 1 - 3 S(29,r(u))$ and $\Delta(29,r(u))= 2+ \dfrac{3r(u)-1}{\delta(29,r(u))}.$
Observe that 
$r \le r(u)\le r(1)$ and $\delta \ge \delta (29,r(u)) \ge \delta(29,r(1)) >0,$
so we may apply Observation \ref{observation2}.
Hence, we have
$2\Delta W(e)^3 \leq 2 \cdot \Delta(29,r(u)) \cdot 2^{3u}.$
From Observation \ref{observation2}, if
$q \ge 2 \cdot \Delta(29,r(u)) \cdot 2^{3u}$ then $(q,2)\in N_3.$
We obtain
\[
\max \{ 2 \cdot \Delta(29,r(u)) \cdot 2^{3u} \mid  1 \le u \le 9 \} < 7.245 \cdot 10^{10},\]
and we conclude that if $q \geq 7.245 \cdot 10^{10},$ 
then $(q,2)\in N_3.$

Now, suppose that $q < 7.245 \cdot 10^{10}$. Again, let $e$ be the product of the primes less than $19$ (including its powers) which divide $q^2-1$ and let $r$ be the number of primes greater or equal to %than 
$19$ which divide $q^2-1.$ We choose $r$ such that $P(19,r)<(q^2-1)/8<6.5613 \cdot 10^{20},$ since $8$ is a factor of $q^2-1.$ The same argument as before implies that $(q,2) \in N_3$ for $q \ge 6.615\cdot 10^8.$

Repeating this process with $q < 6.615\cdot 10^8$ and $e$ the product of the primes less than $19$ (including its powers) which divide $q^2-1,$ we get
$(q,2) \in N_3$ for $q \ge 3.0024\cdot 10^8.$

For smaller values, we can verify directly the condition $q \geq 2 \Delta W(e)^3.$ In fact, it holds for some value of $e,$ for every pair $(q,2),$ where $q$ is an odd prime power, except for a list of $1373$ values of $q.$ See Appendix \ref{appendixm3n2} for this complete list.

\end{proof}

\subsection{Proof of Theorem \ref{Teorema2}}\label{sectionproof2}

%It only remains to verify the values of $(q,n)$ displayed in the second paragraph of Section \ref{sectionproof} (with $n \ge 7$), and those that are not covered by Lemmas \ref{n6m3}, \ref{n5m3}, \ref{n4m3}, \ref{n3m3}, and \ref{n2m3} (with $2 \le n \le 6$). For $n \ge 3$, we use SageMath to manually search $\alpha$ such that $\alpha,$ $\alpha+\beta$, and $\alpha+2\beta$ are primitive, and either $\alpha,$ or $\alpha+\beta,$ or $\alpha+2\beta$ is normal, where the common difference $\beta$ lies in the extension field $\Fqn^*$ (by switching $\beta$ and $-\beta$, this task runtime is cut in half). The genuine exceptions found in these cases are $(q,n,\beta) \in \{%(5,4,1),$ $(5,4,2),$ $(3,4,1),$ $(3,3,1)
%\}$. For $n = 2$, we also use SageMath to manually search such 3-terms arithmetic progressions, however for large values of $q$ the algorithm runtime became very large, thus we just consider common differences $\beta$ in the base field $\Fq^*$. The genuine exceptions found in this case are $(q,n,\beta) \in \{(,,),$ $(,,),$ $(,,)\}$.

It only remains to verify the values of $(q,n)$ displayed in Table \ref{m3nmaior6} %the second paragraph of Section \ref{sectionproof} 
(with $n \ge 7$), and those that are not covered by Lemmas \ref{n6m3}, \ref{n5m3}, \ref{n4m3}, \ref{n3m3}, and \ref{n2m3} (with $2 \le n \le 6$).
We use Algorithm \ref{alg:two} to search $\alpha$ such that $\alpha,$ $\alpha+\beta,$ and $\alpha+2\beta$ are primitive, and either $\alpha,$ or $\alpha+\beta,$ or $\alpha+2\beta$ is normal, for all $\beta \in \Fq^*$ (by switching $\beta$ and $-\beta$, this task runtime is cut in half). The genuine exceptions found in these cases are $(q,n,\beta)$
displayed in Table \ref{table1}. %and $(q,n,\beta)=(9,3,\beta)$ where $\beta$ is any root of the polynomials $x^2 +2x+2, x^2 +x+2 \in \F_3[x]$.

{\small
\begin{table}[h]
\centering
\begin{tabular}{c|c}
	$(q,n)$ & Values of $\beta$ \\
	\hline
	$(3,2)$ & $\beta \in \F_3^*$ \\ \hline 
	$(5,2)$ & $\beta \in \F_5^*$ \\ \hline
	$(7,2)$ & $\beta \in \{\pm 2, \pm 3\}$ \\ \hline
	$(9,2)$ & $\beta \in \F_9^*$ \\ \hline
	$(11,2)$ & $\beta \in \F_{11}^*$ \\ \hline
	$(13,2)$ & $\beta \in \{\pm 1,\pm 4,\pm 5,\pm 6 \}$ \\ \hline
\end{tabular}
\hspace{0.3cm}
\begin{tabular}{c|c}
	$(q,n)$ & Values of $\beta$ \\ \hline
	$(3,3)$ & $\beta \in \F_{3}^*$ \\ \hline
	 	  & $\beta$ is any root \\
	$(9,3)$ & of the polynomials \\
		  & $x^2+2x+2, x^2+x+2 \in \F_3[x]$ \\ \hline
	$(3,4)$ & $\beta \in \F_{3}^*$ \\ \hline
	$(5,4)$ & $\beta \in \F_{5}^*$ \\ \hline 
\end{tabular}
\vspace*{0.5cm}
	\caption{Genuine exceptions for $\beta \in \Fq^*,$ where $\beta$ is a common difference of a 3-terms arithmetic progression formed by primitive elements, and one of the terms is normal.}
	\label{table1}
\end{table}
}
%and 
%$(q,n,\beta)=(9,3,\beta)$ where $\beta$ is any root of the polynomials
%$x^2 +2x+2, x^2 +x+2 \in \F_3[x]$.
%
%\begin{table}[h]
%\centering
%\begin{tabular}{cc}
%	$(q,n)$ & Values of $\beta$ \\
%	\hline
%	$(3,2)$ & $\beta \in \F_3^*$ \\ \hline 
%	$(5,2)$ & $\beta \in \F_5^*$ \\ \hline
%	$(7,2)$ & $\beta \in \{\pm 2, \pm 3\}$ \\ \hline
%\end{tabular} \hspace{0.3cm}
%\begin{tabular}{cc}
%	$(q,n)$ & Values of $\beta$ \\
%	\hline
%	$(9,2)$ & $\beta \in \F_9^*$ \\ \hline
%	$(11,2)$ & $\beta \in \F_{11}^*$ \\ \hline
%	$(13,2)$ & $\beta \in \{\pm 1,\pm 4,\pm 5,\pm 6 \}$ \\ \hline
%\end{tabular}
%\hspace{0.3cm}
%\begin{tabular}{cc}
%	$(q,n)$ & Values of $\beta$ \\
%	\hline
%	$(3,3)$ & $\beta \in \F_{3}^*$ \\ \hline
%	$(3,4)$ & $\beta \in \F_{3}^*$ \\ \hline
%	$(5,4)$ & $\beta \in \F_{5}^*$ \\ \hline 
%\end{tabular}
%\vspace*{0.5cm}
%	\caption{Genuine exceptions for $\beta \in \Fq^*,$ searching 3-terms arithmetic progressions of common ration $\beta$ formed by primitive elements, and one of the terms is normal.}
%	\label{table1}
%\end{table}

\subsection{Proof of Corollary \ref{3primitivospa}}

This corollary follows from Theorem \ref{Teorema2}, by verifying the exceptions. We use Algorithm \ref{alg:two}, removing lines 
\ref{L15} and \ref{L17},
 to search $\alpha$ such that $\alpha,$ $\alpha+\beta$, and $\alpha+2\beta$ are primitive, for all $\beta \in \Fq^*$ (by switching $\beta$ and $-\beta$, this task runtime is cut in half).
The genuine exceptions found in these cases are $(q,n,\beta)$
displayed in Table \ref{tablecorol1}. %and $(q,n,\beta)=(9,3,\beta)$ where $\beta$ is any root of the polynomials $x^2 +2x+2, x^2 +x+2 \in \F_3[x]$.

{\small
\begin{table}[h]
\centering
\begin{tabular}{c|c}
	$(q,n)$ & Values of $\beta$ \\
	\hline
	$(3,2)$ & $\beta \in \F_3^*$ \\ \hline 
	$(5,2)$ & $\beta \in \F_5^*$ \\ \hline
	$(7,2)$ & $\beta \in \{\pm 2, \pm 3\}$ \\ \hline
	$(9,2)$ & $\beta \in \F_9^*$ \\ \hline
	$(11,2)$ & $\beta \in \F_{11}^*$ \\ \hline
\end{tabular} 
\hspace{0.3cm}
\begin{tabular}{c|c}
	$(q,n)$ & Values of $\beta$ \\
	\hline
	$(13,2)$ & $\beta \in \{\pm 1,\pm 4,\pm 5,\pm 6 \}$ \\ \hline
	 	  & $\beta$ is any root \\
	$(9,3)$ & of the polynomials \\
		  & $x^2+2x+2, x^2+x+2 \in \F_3[x]$ \\ \hline
	$(3,4)$ & $\beta \in \F_3^*$ \\ \hline
\end{tabular}
\vspace*{0.5cm}
	\caption{Genuine exceptions for $\beta \in \Fq^*,$ where $\beta$ is a common difference of a 3-terms arithmetic progression formed by primitive elements.}
	\label{tablecorol1}
\end{table}
}

%The genuine exceptions found in these cases are $(q,n,\beta)$
%displayed in Table \ref{table1} and 
%$(q,n,\beta)=(9,3,\beta)$ where $\beta$ is any root of the polynomials
%$x^2 +2x+2, x^2 +x+2 \in \F_3[x]$.
%
%\begin{table}[h]
%\centering
%\begin{tabular}{cc}
%	$(q,n)$ & Values of $\beta$ \\
%	\hline
%	$(3,2)$ & $\beta = \pm1$ \\ \hline 
%	$(5,2)$ & $\beta \in \F_5^*$ \\ \hline
%	$(7,2)$ & $\beta \in \{\pm 2, \pm 3\}$ \\ \hline
%	$(9,2)$ & $\beta \in \F_9^*$ \\ \hline
%\end{tabular} 
%\hspace{0.3cm}
%\begin{tabular}{cc}
%	$(q,n)$ & Values of $\beta$ \\
%	\hline
%	$(11,2)$ & $\beta \in \F_{11}^*$ \\ \hline
%	$(13,2)$ & $\beta \in \{\pm 1,\pm 4,\pm 5,\pm 6 \}$ \\ \hline
%	%$(9,3)$ & $\beta$ is any root of $x^2+2x+2,$ or of $x^2+x+2 \in \F_3[x]$ \\ \hline
%	$(3,4)$ & $\beta = \pm1$ \\ \hline
%\end{tabular}
%\vspace*{0.5cm}
%	\caption{Genuine exceptions for $\beta \in \Fq^*,$ searching 3-terms arithmetic progressions of common ration $\beta$ formed by primitive elements.}
%	\label{table1}
%\end{table}

% \vspace{-0.3cm}

\section{The case $m=2$}\label{sec:m=2}

We now deal with the case $m=2$. By Observation \ref{obs:encaixe}, we have that $N_3 \subset N_2$. Thus, we just need to check the values $(q,n)$ with $q$ odd that fail in Theorems \ref{Teorema1}(ii) and \ref{Teorema2}, in addition to the powers of 2. By Proposition \ref{bound2-1}, for $q$ even, it is only required to verify the powers of 2 up to $7.51 \cdot 10^{358}.$
%\vspace{-0.4cm}
\subsection{Proof of Theorem \ref{Teorema1}(i) for $q$ odd}
We verify the inequality $q^{\frac n2} \ge 2 \Delta W(g) W(e)^2$ (see Proposition \ref{sieve-prop}) for the 1532 %=18+15+8+55+63+1373 
pairs $(q,n)$ that possibly fail in Theorem \ref{Teorema1}(ii), %Lemmas \ref{n6m3}, \ref{n5m3}, \ref{n4m3}, \ref{n3m3}, and \ref{n2m3}
considering every divisor $g$ of $x^n - 1$ and every divisor $e$ of $q^n - 1$. From these, all but the 155 %=111+44
pairs displayed in Table \ref{tablem2qodd} work for this test.

{\small
\begin{table}[H]
\centering
\begin{tabular}{c|c}
$n$ & $q$ \\
\hline
$12$ & $5,3$ \\ \hline 
$10$ & $3$ \\ \hline 
$8$ & $9,5,3$ \\ \hline 
$7$ & $3$ \\ \hline 
$6$ & $11,9,7,5,3$ \\ \hline 
$5$ & $11,7,5,3$ \\ \hline 
 & $43,41,29,23,$ \\
$4$ & $19,17,13,11,$ \\
 & $9,7,5,3$ \\ \hline 
 & $211,121,67,$ \\ 
$3$ & $61,43,37,31,$ \\
 & $25,23,19,13,$ \\
 & $11,9,7,5,3$ \\ \hline 
\end{tabular}
\hspace{0.3cm}
\begin{tabular}{c|c}
$n$ & $q$ \\
\hline
 & $3191, 2729, 2311, 2029, 1871, 1861,1849,$ \\ 
 & $1709, 1429, 1331, 1301, 1289, 1259,1231,$ \\ 
 & $1091, 1021, 911, 881, 859, 811,769, 701,$ \\ 
 & $659, 631, 601, 599, 571, 529,521, 509,$ \\ 
 & $491, 463, 461, 449, 439, 421,419, 409,$ \\ 
 & $389, 379, 373, 349, 337,331, 311, 307,$ \\ 
$2$ & $293, 289, 281, 271, 269,263, 251, 241,$ \\ 
 & $239, 233, 229, 227,223, 211, 199, 197,$ \\ 
 & $191, 181, 179, 173,169, 167, 157, 151,$ \\ 
 & $149, 139, 137,131, 127, 125, 121, 113,$ \\ 
 & $109, 107, 103,101, 97, 89, 83, 81, 79,73,$ \\ 
 & $71, 67,61, 59, 53, 49, 47, 43, 41,37, 31,$ \\ 
 & $29, 27, 25, 23, 19, 17, 13, 11,9, 7, 5, 3$ \\ \hline
\end{tabular} %\hspace{0.3cm}
\vspace*{0.5cm}
	\caption{Possible exceptions for Theorem \ref{Teorema1}(i) with $q$ odd.}
	\label{tablem2qodd}
\end{table}
}

%\begin{table}[H]
%\centering
%\begin{tabular}{c|l}
%$n$ & $q$ \\
%\hline
%$12$ & $5,3$ \\ \hline 
%$10$ & $3$ \\ \hline 
%$8$ & $9,5,3$ \\ \hline 
%$7$ & $3$ \\ \hline 
%$6$ & $11,9,7,5,3$ \\ \hline 
%$5$ & $11,7,5,3$ \\ \hline 
%$4$ & $43,41,29,23,19,17,13,11,9,7,5,3$ \\ \hline 
%$3$ & $211,121,67,61,43,37,31,25,23,19,13,11,0,7,5,3$ \\ \hline 
% & $3191, 2729, 2311, 2029, 1871, 1861, 1849, 1709, 1429, 1331, 1301, 1289, 1259,$ \\ 
% & $1231, 1091, 1021, 911, 881, 859, 811, 769, 701, 659, 631, 601, 599, 571, 529,$ \\  
% & $521, 509, 491, 463, 461, 449, 439, 421, 419, 409, 389, 379, 373, 349, 337,$ \\  
%$2$ & $331, 311, 307, 293, 289, 281, 271, 269, 263, 251, 241, 239, 233, 229, 227,$ \\  
% & $223, 211, 199, 197, 191, 181, 179, 173, 169, 167, 157, 151, 149, 139, 137,$ \\  
% & $131, 127, 125, 121, 113, 109, 107, 103, 101, 97, 89, 83, 81, 79, 73, 71, 67,$ \\  
% & $61, 59, 53, 49, 47, 43, 41, 37, 31, 29, 27, 25, 23, 19, 17, 13, 11, 9, 7, 5, 3$ \\ \hline
%\end{tabular} %\hspace{0.3cm}
%\vspace*{0.5cm}
%	\caption{Possibly exceptions for Theorem \ref{Teorema1}(i) with $q$ odd}
%	\label{tablem2qodd}
%\end{table}

\subsection{Proof of Theorem \ref{Teorema1}(i) for $q$ even}\label{teo1iqeven}

%Let $q = 2^k,$ $n \ge 2,$ and $1 \le k \le 596,$ so that $2^2 \le (2^k)^2 \le q^2 \le q^n < 7.51 \cdot 10^{358}.$ We apply the first of the following methods for the pairs $(q,n)$ and the second one for those pairs where the first method fails.
Let $q = 2^k$ with $1 \le k \le 596,$ and let $n \ge 2,$ such that $q^n < 7.51 \cdot 10^{358}.$ We apply the first of the following methods for the pairs $(q,n)$ and the second one for those pairs where the first method fails.

\begin{itemize}
\item By Lemma \ref{lem:3.7}, it follows that for $q=2$ it holds $W(x^n-1) \le 2^{\frac{n+14}{5}},$ for $q=4$ it holds $W(x^n-1) \le 2^{\frac{n+41}{4}},$ for $q=8$ it holds $W(x^n-1) \le 2^{\frac{n}{3}+14},$ for $q=16$ it holds $W(x^n-1) \le 2^{\frac{n+15}{2}}$, and for $q \ge 32$ it holds $W(x^n-1) \le 2^n.$ Taking into account that $q^n-1$ is odd, the constant $A_t$ given by Lemma \ref{cota-t} does not consider the prime $\wp = 2.$ We then verify if $q^{\frac n2} \ge 2W(x^n-1) A_t^2 q^{\frac{2n}{t}}$ using the bounds for $A_t$ and for $W(x^n-1)$, 
%W(M) \leq A_t \cdot M^{\frac{1}{t}}
and apply Theorem \ref{principal}.

\item Let $p_0$ be a fixed prime number, $L$ be the product of primes up to $p_0$ that divide $q^n-1,$ and $T = \frac{q^n-1}{\lcm(q^n-1,L)}.$ Let $s \in \mathbb N$ be the largest possible so that $p_0 < p_1 < \dots < p_s$ are consecutive primes with $p_0 p_1 \dots p_s \le T,$ and $\delta \ge \delta_T := 1 - 2 \sum_{i=0}^s \frac{1}{p_i} > 0.$ We then verify if $q^{\frac{n}{2}} \geq m \Delta_T W(x^n-1) W(e)^m,$ where $\Delta_T := 2+\frac{2s-1}{\delta_T} \ge \Delta,$ and apply Proposition \ref{sieve-prop}.
\end{itemize}

%If we choose $t=10$ and $p_0=61,$ then at least one of the previous methods work for the pairs $(q,n) = (2, n)$ with $n > 36,$ $(4,n)$ with $n > 18,$ $(8,n)$ with $n > 8,$ $(16, n)$ with $n > 15,$ $(32,n)$ with $n > 4,$ $(64, n)$ with $n > 6,$ $(128, n)$ with $n > 2,$ $(256,n)$ with $n > 5,$ $(512, n)$ with $n > 2,$ $(1024, n)$ with $n > 2,$ $(4096, n)$ with $n > 3,$ and $(q, n)$ with $q = 2^k > 4096.$

If we choose $t=10$ and $p_0=61,$ then at least one of the previous methods work except for the pairs displayed in Table \ref{tablem2qeven}.

{\small
\begin{table}[H]
\centering
\begin{tabular}{c|c}
$q$ & $n$ \\
\hline
 & $36,30,28,24,21,20,$ \\
$2$ & $18,16,15,14,12,11,$ \\
 & $10,9,8,7,6,5,4,3,2$ \\ \hline
$4$ & $18,15,14,12,10,9,$ \\
 & $8,7,6,5,4,3,2$ \\ \hline
\end{tabular}
\hspace{0.3cm}
\begin{tabular}{c|c}
$q$ & $n$ \\
\hline
$8$ & $8,7,6,5,4,3,2$ \\ \hline
$16$ & $15,10,9,7,6,5,$ \\
 & $4,3,2$ \\ \hline
$32$ & $4,3,2$ \\ \hline
$64$ & $6,4,3,2$ \\ \hline
\end{tabular}
\hspace{0.3cm}
\begin{tabular}{c|c}
$q$ & $n$ \\
\hline
$128$ & $2$ \\ \hline
$256$ & $5,3,2$ \\ \hline
$512$ & $2$ \\ \hline
$1024$ & $2$ \\ \hline
$4096$ & $3,2$ \\ \hline
%$(2, 36),$
% $(2, 30),$
% $(2, 28),$
% $(2, 24),$
% $(2, 21),$
% $(2, 20),$
% $(2, 18),$
% $(2, 16),$
% $(2, 15),$
% $(2, 14),$
% $(2, 12),$
% $(2, 11),$
% $(2, 10),$
% $(2, 9),$
% $(2, 8),$
% $(2, 7),$
% $(2, 6),$
% $(2, 5),$
% $(2, 4),$
% $(2, 3),$
% $(2, 2),$
% $(4, 18),$
% $(4, 15),$
% $(4, 14),$
% $(4, 12),$
% $(4, 10),$
% $(4, 9),$
% $(4, 8),$
% $(4, 7),$
% $(4, 6),$
% $(4, 5),$
% $(4, 4),$
% $(4, 3),$
% $(4, 2),$
% $(8, 8),$
% $(8, 7),$
% $(8, 6),$
% $(8, 5),$
% $(8, 4),$
% $(8, 3),$
% $(8, 2),$
% $(16, 15),$
% $(16, 10),$
% $(16, 9),$
% $(16, 7),$
% $(16, 6),$
% $(16, 5),$
% $(16, 4),$
% $(16, 3),$
% $(16, 2),$
% $(32, 4),$
% $(32, 3),$
% $(32, 2),$
% $(64, 6),$
% $(64, 4),$
% $(64, 3),$
% $(64, 2),$
% $(128, 2),$
% $(256, 5),$
% $(256, 3),$
% $(256, 2),$
% $(512, 2),$
% $(1024, 2),$
% $(4096, 3),$
% $(4096, 2).$
\end{tabular} %\hspace{0.3cm}
\vspace*{0.5cm}
	\caption{Partial possible exceptions for Theorem \ref{Teorema1}(i) with $q$ even.}
	\label{tablem2qeven}
\end{table}
}

We verify the inequality $q^{\frac n2} \ge 2 \Delta W(g) W(e)^2$ (see Proposition \ref{sieve-prop}) for the pairs $(q,n)$ of Table \ref{tablem2qeven}, considering every divisor $g$ of $x^n - 1$ and every divisor $e$ of $q^n - 1$. From these, all but the pairs displayed in Table \ref{tablem2qevenfinal} work for this test.

{\small
\begin{table}[H]
\centering
\begin{tabular}{c|c}
$q$ & $n$ \\
\hline
$2$ & $14,12,10,9,8,$ \\
 & $7,6,5,4,3,2$ \\ \hline
$4$ & $12,9,6,5,$ \\
 & $4,3,2$ \\ \hline
\end{tabular}
\hspace{0.3cm}
\begin{tabular}{c|c}
$q$ & $n$ \\
\hline
$8$ & $4,2$ \\ \hline
$16$ & $3,2$ \\ \hline
$32$ & $2$ \\ \hline
$64$ & $2$ \\ \hline
\end{tabular} %\hspace{0.3cm}
\vspace*{0.5cm}
	\caption{Possible exceptions for Theorem \ref{Teorema1}(i) with $q$ even.}
	\label{tablem2qevenfinal}
\end{table}
}
%all but the following pairs $(q,n)$: $(2, 14),$ $(2, 12),$ $(2, 10),$ $(2, 9),$ $(2, 8),$ $(2, 7),$ $(2, 6),$ $(2, 5),$ $(2, 4),$ $(2, 3),$ $(2, 2),$ $(4, 12),$ $(4, 9),$ $(4, 6),$ $(4, 5),$ $(4, 4),$ $(4, 3),$ $(4, 2),$ $(8, 4),$ $(8, 2),$ $(16, 3),$ $(16, 2),$ $(32, 2),$ $(64, 2).$ It completes the proof of Theorem \ref{Teorema1}. 

\subsection{Proof of Theorem \ref{Teorema3}}

It only remains to verify the values of $(q,n)$ displayed in Tables \ref{tablem2qodd} and \ref{tablem2qevenfinal}. We remove line \ref{L13} and adapt lines \ref{L11}, \ref{L14} and \ref{L15} of Algorithm \ref{alg:two} for $q$ odd,
and we also adapt line \ref{L4} in the case where $q$ is even, to search a primitive element $\alpha \in \Fqn$ such that $\alpha+\beta$ is also primitive and either $\alpha$ or $\alpha+\beta$ is normal.  
The only genuine exception found in this case is $(q,n,\beta) = (2,4,1).$ It completes the proof of Theorem \ref{Teorema3}.

\subsection{Proof of Corollary \ref{2primitivospa}}

This corollary follows directly from Theorem \ref{Teorema3} and \cite[Theorem A]{Cohen1985c}, since for $\beta=1$ we deal with consecutive elements.

\appendix \section{Possible exceptions for Theorem \ref{Teorema1}(ii) with $n=2$}\label{appendixm3n2}

In this appendix, it is displayed the 1373 values of $q$ that are possible exceptions for Theorem \ref{Teorema1}(ii) with $n=2$; see Lemma \ref{n2m3}.

%{\singlespacing
\noindent {\scriptsize %\small %\footnotesize 
$3,$ 
 $5,$ 
 $7,$ 
 $9,$ 
 $11,$ 
 $13,$ 
 $17,$ 
 $19,$ 
 $23,$ 
 $25,$ 
 $27,$ 
 $29,$ 
 $31,$ 
 $37,$ 
 $41,$ 
 $43,$ 
 $47,$ 
 $49,$ 
 $53,$ 
 $59,$ 
 $61,$ 
 $67,$ 
 $71,$ 
 $73,$ 
 $79,$ 
 $81,$ 
 $83,$ 
 $89,$ 
 $97,$ 
 $101,$ 
 $103,$ 
 $107,$ 
 $109,$ 
 $113,$ 
 $121,$ 
 $125,$ 
 $127,$ 
 $131,$ 
 $137,$ 
 $139,$ 
 $149,$ 
 $151,$ 
 $157,$ 
 $163,$ 
 $167,$ 
 $169,$ 
 $173,$ 
 $179,$ 
 $181,$ 
 $191,$ 
 $193,$ 
 $197,$ 
 $199,$ 
 $211,$ 
 $223,$ 
 $227,$ 
 $229,$ 
 $233,$ 
 $239,$ 
 $241,$ 
 $251,$ 
 $257,$ 
 $263,$ 
 $269,$ 
 $271,$ 
 $277,$ 
 $281,$ 
 $283,$ 
 $289,$ 
 $293,$ 
 $307,$ 
 $311,$ 
 $313,$ 
 $317,$ 
 $331,$ 
 $337,$ 
 $343,$ 
 $347,$ 
 $349,$ 
 $353,$ 
 $359,$ 
 $361,$ 
 $367,$ 
 $373,$ 
 $379,$ 
 $383,$ 
 $389,$ 
 $397,$ 
 $401,$ 
 $409,$ 
 $419,$ 
 $421,$ 
 $431,$ 
 $433,$ 
 $439,$ 
 $443,$ 
 $449,$ 
 $457,$ 
 $461,$ 
 $463,$ 
 $467,$ 
 $479,$ 
 $487,$ 
 $491,$ 
 $499,$ 
 $503,$ 
 $509,$ 
 $521,$ 
 $523,$ 
 $529,$ 
 $541,$ 
 $547,$ 
 $557,$ 
 $563,$ 
 $569,$ 
 $571,$ 
 $587,$ 
 $593,$ 
 $599,$ 
 $601,$ 
 $607,$ 
 $613,$ 
 $617,$ 
 $619,$ 
 $625,$ 
 $631,$ 
 $641,$ 
 $643,$ 
 $647,$ 
 $653,$ 
 $659,$ 
 $661,$ 
 $673,$ 
 $677,$ 
 $683,$ 
 $691,$ 
 $701,$ 
 $709,$ 
 $719,$ 
 $727,$ 
 $729,$ 
 $733,$ 
 $739,$ 
 $743,$ 
 $751,$ 
 $757,$ 
 $761,$ 
 $769,$ 
 $773,$ 
 $787,$ 
 $797,$ 
 $809,$ 
 $811,$ 
 $821,$ 
 $823,$ 
 $827,$ 
 $829,$ 
 $839,$ 
 $841,$ 
 $853,$ 
 $857,$ 
 $859,$ 
 $877,$ 
 $881,$ 
 $883,$ 
 $887,$ 
 $907,$ 
 $911,$ 
 $919,$ 
 $929,$ 
 $937,$ 
 $941,$ 
 $947,$ 
 $953,$ 
 $961,$ 
 $967,$ 
 $971,$ 
 $991,$ 
 $1009,$ 
 $1013,$ 
 $1019,$ 
 $1021,$ 
 $1031,$ 
 $1033,$ 
 $1039,$ 
 $1049,$ 
 $1051,$ 
 $1061,$ 
 $1063,$ 
 $1069,$ 
 $1091,$ 
 $1093,$ 
 $1103,$ 
 $1109,$ 
 $1117,$ 
 $1123,$ 
 $1129,$ 
 $1151,$ 
 $1163,$ 
 $1171,$ 
 $1181,$ 
 $1201,$ 
 $1217,$ 
 $1223,$ 
 $1229,$ 
 $1231,$ 
 $1249,$ 
 $1259,$ 
 $1277,$ 
 $1279,$ 
 $1289,$ 
 $1291,$ 
 $1301,$ 
 $1303,$ 
 $1319,$ 
 $1321,$ 
 $1327,$ 
 $1331,$ 
 $1361,$ 
 $1369,$ 
 $1373,$ 
 $1381,$ 
 $1399,$ 
 $1409,$ 
 $1427,$ 
 $1429,$ 
 $1439,$ 
 $1451,$ 
 $1459,$ 
 $1471,$ 
 $1481,$ 
 $1483,$ 
 $1489,$ 
 $1499,$ 
 $1511,$ 
 $1531,$ 
 $1549,$ 
 $1553,$ 
 $1559,$ 
 $1567,$ 
 $1571,$ 
 $1579,$ 
 $1583,$ 
 $1597,$ 
 $1601,$ 
 $1607,$ 
 $1609,$ 
 $1613,$ 
 $1619,$ 
 $1621,$ 
 $1627,$ 
 $1637,$ 
 $1667,$ 
 $1669,$ 
 $1681,$ 
 $1693,$ 
 $1699,$ 
 $1709,$ 
 $1721,$ 
 $1723,$ 
 $1741,$ 
 $1747,$ 
 $1759,$ 
 $1777,$ 
 $1789,$ 
 $1801,$ 
 $1811,$ 
 $1831,$ 
 $1847,$ 
 $1849,$ 
 $1861,$ 
 $1871,$ 
 $1877,$ 
 $1879,$ 
 $1889,$ 
 $1901,$ 
 $1913,$ 
 $1931,$ 
 $1933,$ 
 $1949,$ 
 $1951,$ 
 $1973,$ 
 $1979,$ 
 $1987,$ 
 $1999,$ 
 $2003,$ 
 $2011,$ 
 $2029,$ 
 $2039,$ 
 $2069,$ 
 $2081,$ 
 $2087,$ 
 $2089,$ 
 $2099,$ 
 $2111,$ 
 $2113,$ 
 $2129,$ 
 $2131,$ 
 $2141,$ 
 $2143,$ 
 $2161,$ 
 $2179,$ 
 $2197,$ 
 $2209,$ 
 $2213,$ 
 $2221,$ 
 $2239,$ 
 $2243,$ 
 $2267,$ 
 $2269,$ 
 $2281,$ 
 $2287,$ 
 $2297,$ 
 $2309,$ 
 $2311,$ 
 $2339,$ 
 $2341,$ 
 $2351,$ 
 $2371,$ 
 $2381,$ 
 $2389,$ 
 $2393,$ 
 $2399,$ 
 $2411,$ 
 $2437,$ 
 $2441,$ 
 $2459,$ 
 $2521,$ 
 $2531,$ 
 $2539,$ 
 $2549,$ 
 $2551,$ 
 $2579,$ 
 $2591,$ 
 $2609,$ 
 $2617,$ 
 $2621,$ 
 $2659,$ 
 $2671,$ 
 $2687,$ 
 $2689,$ 
 $2699,$ 
 $2711,$ 
 $2719,$ 
 $2729,$ 
 $2731,$ 
 $2741,$ 
 $2749,$ 
 $2789,$ 
 $2791,$ 
 $2801,$ 
 $2809,$ 
 $2819,$ 
 $2851,$ 
 $2857,$ 
 $2861,$ 
 $2909,$ 
 $2927,$ 
 $2939,$ 
 $2969,$ 
 $2971,$ 
 $3001,$ 
 $3011,$ 
 $3037,$ 
 $3041,$ 
 $3049,$ 
 $3061,$ 
 $3067,$ 
 $3079,$ 
 $3089,$ 
 $3109,$ 
 $3119,$ 
 $3121,$ 
 $3163,$ 
 $3169,$ 
 $3181,$ 
 $3191,$ 
 $3221,$ 
 $3229,$ 
 $3251,$ 
 $3299,$ 
 $3301,$ 
 $3319,$ 
 $3329,$ 
 $3331,$ 
 $3359,$ 
 $3361,$ 
 $3389,$ 
 $3391,$ 
 $3433,$ 
 $3449,$ 
 $3457,$ 
 $3469,$ 
 $3481,$ 
 $3499,$ 
 $3511,$ 
 $3529,$ 
 $3539,$ 
 $3541,$ 
 $3571,$ 
 $3613,$ 
 $3631,$ 
 $3659,$ 
 $3671,$ 
 $3691,$ 
 $3697,$ 
 $3709,$ 
 $3719,$ 
 $3739,$ 
 $3761,$ 
 $3769,$ 
 $3779,$ 
 $3821,$ 
 $3851,$ 
 $3877,$ 
 $3911,$ 
 $3919,$ 
 $3989,$ 
 $4001,$ 
 $4003,$ 
 $4019,$ 
 $4049,$ 
 $4079,$ 
 $4091,$ 
 $4129,$ 
 $4159,$ 
 $4201,$ 
 $4211,$ 
 $4219,$ 
 $4229,$ 
 $4231,$ 
 $4241,$ 
 $4271,$ 
 $4289,$ 
 $4339,$ 
 $4409,$ 
 $4421,$ 
 $4423,$ 
 $4451,$ 
 $4481,$ 
 $4489,$ 
 $4523,$ 
 $4549,$ 
 $4591,$ 
 $4621,$ 
 $4649,$ 
 $4663,$ 
 $4679,$ 
 $4691,$ 
 $4729,$ 
 $4751,$ 
 $4759,$ 
 $4789,$ 
 $4801,$ 
 $4817,$ 
 $4831,$ 
 $4861,$ 
 $4871,$ 
 $4889,$ 
 $4931,$ 
 $4951,$ 
 $4969,$ 
 $4999,$ 
 $5011,$ 
 $5039,$ 
 $5041,$ 
 $5059,$ 
 $5081,$ 
 $5167,$ 
 $5171,$ 
 $5179,$ 
 $5209,$ 
 $5237,$ 
 $5279,$ 
 $5281,$ 
 $5329,$ 
 $5381,$ 
 $5419,$ 
 $5431,$ 
 $5479,$ 
 $5501,$ 
 $5519,$ 
 $5521,$ 
 $5531,$ 
 $5591,$ 
 $5641,$ 
 $5659,$ 
 $5669,$ 
 $5711,$ 
 $5741,$ 
 $5839,$ 
 $5849,$ 
 $5851,$ 
 $5879,$ 
 $5881,$ 
 $5939,$ 
 $5981,$ 
 $6007,$ 
 $6029,$ 
 $6089,$ 
 $6091,$ 
 $6131,$ 
 $6203,$ 
 $6221,$ 
 $6229,$ 
 $6241,$ 
 $6269,$ 
 $6271,$ 
 $6299,$ 
 $6301,$ 
 $6329,$ 
 $6359,$ 
 $6379,$ 
 $6421,$ 
 $6449,$ 
 $6469,$ 
 $6481,$ 
 $6491,$ 
 $6551,$ 
 $6553,$ 
 $6571,$ 
 $6581,$ 
 $6599,$ 
 $6679,$ 
 $6689,$ 
 $6691,$ 
 $6709,$ 
 $6719,$ 
 $6733,$ 
 $6761,$ 
 $6791,$ 
 $6841,$ 
 $6859,$ 
 $6869,$ 
 $6889,$ 
 $6917,$ 
 $6959,$ 
 $6971,$ 
 $6991,$ 
 $7001,$ 
 $7019,$ 
 $7039,$ 
 $7069,$ 
 $7129,$ 
 $7151,$ 
 $7211,$ 
 $7229,$ 
 $7237,$ 
 $7253,$ 
 $7309,$ 
 $7321,$ 
 $7331,$ 
 $7349,$ 
 $7351,$ 
 $7369,$ 
 $7411,$ 
 $7459,$ 
 $7481,$ 
 $7489,$ 
 $7541,$ 
 $7547,$ 
 $7549,$ 
 $7559,$ 
 $7561,$ 
 $7589,$ 
 $7591,$ 
 $7639,$ 
 $7669,$ 
 $7699,$ 
 $7741,$ 
 $7789,$ 
 $7829,$ 
 $7841,$ 
 $7853,$ 
 $7879,$ 
 $7919,$ 
 $7921,$ 
 $7951,$ 
 $8009,$ 
 $8059,$ 
 $8161,$ 
 $8171,$ 
 $8191,$ 
 $8219,$ 
 $8231,$ 
 $8269,$ 
 $8329,$ 
 $8429,$ 
 $8431,$ 
 $8501,$ 
 $8513,$ 
 $8527,$ 
 $8539,$ 
 $8581,$ 
 $8609,$ 
 $8669,$ 
 $8681,$ 
 $8689,$ 
 $8737,$ 
 $8741,$ 
 $8761,$ 
 $8779,$ 
 $8819,$ 
 $8821,$ 
 $8839,$ 
 $8849,$ 
 $8861,$ 
 $8929,$ 
 $8969,$ 
 $8971,$ 
 $9001,$ 
 $9029,$ 
 $9041,$ 
 $9043,$ 
 $9049,$ 
 $9059,$ 
 $9109,$ 
 $9151,$ 
 $9199,$ 
 $9239,$ 
 $9241,$ 
 $9281,$ 
 $9283,$ 
 $9311,$ 
 $9349,$ 
 $9371,$ 
 $9409,$ 
 $9421,$ 
 $9437,$ 
 $9439,$ 
 $9461,$ 
 $9463,$ 
 $9479,$ 
 $9491,$ 
 $9521,$ 
 $9547,$ 
 $9619,$ 
 $9631,$ 
 $9661,$ 
 $9689,$ 
 $9769,$ 
 $9791,$ 
 $9811,$ 
 $9829,$ 
 $9857,$ 
 $9859,$ 
 $9871,$ 
 $9931,$ 
 $9941,$ 
 $10009,$ 
 $10039,$ 
 $10061,$ 
 $10079,$ 
 $10099,$ 
 $10139,$ 
 $10141,$ 
 $10151,$ 
 $10259,$ 
 $10271,$ 
 $10289,$ 
 $10321,$ 
 $10331,$ 
 $10429,$ 
 $10459,$ 
 $10499,$ 
 $10501,$ 
 $10529,$ 
 $10601,$ 
 $10609,$ 
 $10639,$ 
 $10709,$ 
 $10711,$ 
 $10739,$ 
 $10781,$ 
 $10789,$ 
 $10891,$ 
 $10949,$ 
 $10979,$ 
 $11059,$ 
 $11131,$ 
 $11159,$ 
 $11171,$ 
 $11299,$ 
 $11311,$ 
 $11329,$ 
 $11351,$ 
 $11369,$ 
 $11411,$ 
 $11491,$ 
 $11549,$ 
 $11551,$ 
 $11579,$ 
 $11593,$ 
 $11621,$ 
 $11681,$ 
 $11689,$ 
 $11719,$ 
 $11731,$ 
 $11779,$ 
 $11789,$ 
 $11801,$ 
 $11831,$ 
 $11881,$ 
 $11941,$ 
 $11959,$ 
 $11969,$ 
 $11971,$ 
 $12011,$ 
 $12041,$ 
 $12109,$ 
 $12167,$ 
 $12211,$ 
 $12239,$ 
 $12391,$ 
 $12401,$ 
 $12409,$ 
 $12451,$ 
 $12479,$ 
 $12511,$ 
 $12539,$ 
 $12541,$ 
 $12641,$ 
 $12671,$ 
 $12689,$ 
 $12739,$ 
 $12769,$ 
 $12781,$ 
 $12791,$ 
 $12809,$ 
 $12919,$ 
 $12959,$ 
 $12979,$ 
 $13001,$ 
 $13049,$ 
 $13109,$ 
 $13159,$ 
 $13259,$ 
 $13331,$ 
 $13339,$ 
 $13397,$ 
 $13399,$ 
 $13411,$ 
 $13421,$ 
 $13441,$ 
 $13469,$ 
 $13649,$ 
 $13679,$ 
 $13691,$ 
 $13729,$ 
 $13789,$ 
 $13831,$ 
 $13859,$ 
 $13901,$ 
 $13931,$ 
 $14029,$ 
 $14071,$ 
 $14249,$ 
 $14251,$ 
 $14281,$ 
 $14321,$ 
 $14419,$ 
 $14431,$ 
 $14449,$ 
 $14461,$ 
 $14489,$ 
 $14519,$ 
 $14561,$ 
 $14629,$ 
 $14741,$ 
 $14771,$ 
 $14821,$ 
 $14851,$ 
 $14869,$ 
 $14939,$ 
 $14951,$ 
 $15091,$ 
 $15131,$ 
 $15149,$ 
 $15161,$ 
 $15259,$ 
 $15289,$ 
 $15329,$ 
 $15331,$ 
 $15391,$ 
 $15401,$ 
 $15443,$ 
 $15511,$ 
 $15541,$ 
 $15569,$ 
 $15581,$ 
 $15619,$ 
 $15641,$ 
 $15679,$ 
 $15731,$ 
 $15749,$ 
 $15791,$ 
 $15809,$ 
 $15889,$ 
 $15919,$ 
 $15959,$ 
 $16141,$ 
 $16301,$ 
 $16339,$ 
 $16381,$ 
 $16421,$ 
 $16451,$ 
 $16519,$ 
 $16561,$ 
 $16619,$ 
 $16631,$ 
 $16661,$ 
 $16729,$ 
 $16759,$ 
 $16829,$ 
 $16831,$ 
 $16871,$ 
 $17029,$ 
 $17137,$ 
 $17159,$ 
 $17161,$ 
 $17291,$ 
 $17341,$ 
 $17359,$ 
 $17389,$ 
 $17401,$ 
 $17471,$ 
 $17569,$ 
 $17579,$ 
 $17599,$ 
 $17669,$ 
 $17681,$ 
 $17851,$ 
 $17863,$ 
 $17921,$ 
 $18041,$ 
 $18059,$ 
 $18061,$ 
 $18089,$ 
 $18121,$ 
 $18131,$ 
 $18149,$ 
 $18199,$ 
 $18229,$ 
 $18269,$ 
 $18329,$ 
 $18451,$ 
 $18461,$ 
 $18481,$ 
 $18539,$ 
 $18661,$ 
 $18719,$ 
 $18859,$ 
 $18869,$ 
 $18899,$ 
 $19031,$ 
 $19081,$ 
 $19139,$ 
 $19141,$ 
 $19181,$ 
 $19249,$ 
 $19319,$ 
 $19321,$ 
 $19381,$ 
 $19447,$ 
 $19469,$ 
 $19489,$ 
 $19501,$ 
 $19531,$ 
 $19559,$ 
 $19571,$ 
 $19609,$ 
 $19739,$ 
 $19759,$ 
 $19889,$ 
 $19949,$ 
 $19991,$ 
 $20021,$ 
 $20089,$ 
 $20129,$ 
 $20149,$ 
 $20161,$ 
 $20201,$ 
 $20231,$ 
 $20369,$ 
 $20399,$ 
 $20411,$ 
 $20747,$ 
 $20749,$ 
 $20789,$ 
 $20879,$ 
 $20929,$ 
 $21011,$ 
 $21139,$ 
 $21169,$ 
 $21319,$ 
 $21391,$ 
 $21419,$ 
 $21559,$ 
 $21589,$ 
 $21713,$ 
 $21757,$ 
 $21839,$ 
 $21841,$ 
 $21911,$ 
 $22079,$ 
 $22133,$ 
 $22259,$ 
 $22441,$ 
 $22469,$ 
 $22541,$ 
 $22571,$ 
 $22639,$ 
 $22679,$ 
 $22751,$ 
 $22861,$ 
 $22961,$ 
 $23011,$ 
 $23029,$ 
 $23087,$ 
 $23099,$ 
 $23143,$ 
 $23189,$ 
 $23269,$ 
 $23311,$ 
 $23321,$ 
 $23561,$ 
 $23563,$ 
 $23629,$ 
 $23689,$ 
 $23827,$ 
 $23869,$ 
 $23981,$ 
 $24179,$ 
 $24359,$ 
 $24389,$ 
 $24509,$ 
 $24571,$ 
 $24611,$ 
 $24649,$ 
 $24683,$ 
 $24709,$ 
 $24821,$ 
 $24851,$ 
 $25117,$ 
 $25171,$ 
 $25229,$ 
 $25339,$ 
 $25409,$ 
 $25411,$ 
 $25439,$ 
 $25453,$ 
 $25609,$ 
 $25621,$ 
 $25741,$ 
 $25801,$ 
 $25999,$ 
 $26041,$ 
 $26321,$ 
 $26489,$ 
 $26641,$ 
 $26839,$ 
 $26861,$ 
 $26951,$ 
 $27061,$ 
 $27091,$ 
 $27259,$ 
 $27481,$ 
 $27509,$ 
 $27551,$ 
 $27611,$ 
 $27691,$ 
 $28181,$ 
 $28211,$ 
 $28289,$ 
 $28309,$ 
 $28559,$ 
 $28729,$ 
 $28909,$ 
 $29231,$ 
 $29303,$ 
 $29581,$ 
 $29611,$ 
 $29819,$ 
 $30029,$ 
 $30059,$ 
 $30161,$ 
 $30211,$ 
 $30269,$ 
 $30449,$ 
 $30689,$ 
 $30911,$ 
 $30941,$ 
 $31121,$ 
 $31151,$ 
 $31219,$ 
 $31541,$ 
 $31891,$ 
 $32059,$ 
 $32299,$ 
 $32341,$ 
 $32369,$ 
 $32579,$ 
 $32719,$ 
 $32801,$ 
 $32941,$ 
 $32969,$ 
 $33151,$ 
 $33349,$ 
 $33461,$ 
 $33529,$ 
 $33769,$ 
 $33851,$ 
 $34033,$ 
 $34061,$ 
 $34231,$ 
 $34511,$ 
 $34649,$ 
 $34781,$ 
 $35069,$ 
 $35099,$ 
 $35111,$ 
 $35281,$ 
 $35419,$ 
 $35491,$ 
 $35531,$ 
 $35671,$ 
 $35729,$ 
 $35771,$ 
 $35869,$ 
 $36191,$ 
 $36541,$ 
 $36709,$ 
 $36721,$ 
 $36791,$ 
 $36821,$ 
 $36919,$ 
 $37309,$ 
 $37379,$ 
 $37619,$ 
 $38011,$ 
 $38039,$ 
 $38149,$ 
 $38219,$ 
 $38501,$ 
 $38569,$ 
 $38611,$ 
 $38851,$ 
 $39439,$ 
 $39521,$ 
 $39929,$ 
 $40039,$ 
 $40151,$ 
 $40459,$ 
 $40699,$ 
 $40949,$ 
 $41341,$ 
 $41411,$ 
 $41539,$ 
 $41651,$ 
 $41999,$ 
 $42181,$ 
 $42461,$ 
 $42701,$ 
 $42979,$ 
 $43499,$ 
 $43889,$ 
 $43891,$ 
 $44269,$ 
 $44549,$ 
 $44771,$ 
 $45121,$ 
 $45319,$ 
 $45541,$ 
 $46171,$ 
 $46229,$ 
 $46411,$ 
 $46619,$ 
 $47059,$ 
 $47431,$ 
 $47501,$ 
 $47741,$ 
 $48049,$ 
 $48179,$ 
 $48299,$ 
 $48619,$ 
 $49279,$ 
 $49477,$ 
 $49531,$ 
 $49741,$ 
 $49939,$ 
 $50051,$ 
 $50231,$ 
 $51169,$ 
 $51239,$ 
 $51479,$ 
 $51869,$ 
 $51871,$ 
 $52051,$ 
 $52249,$ 
 $52361,$ 
 $53129,$ 
 $53299,$ 
 $53381,$ 
 $53549,$ 
 $53591,$ 
 $53899,$ 
 $54251,$ 
 $54419,$ 
 $54559,$ 
 $54601,$ 
 $54979,$ 
 $55021,$ 
 $55441,$ 
 $55691,$ 
 $55901,$ 
 $55931,$ 
 $56099,$ 
 $56239,$ 
 $56629,$ 
 $56671,$ 
 $57331,$ 
 $58631,$ 
 $59149,$ 
 $59669,$ 
 $60521,$ 
 $60719,$ 
 $60761,$ 
 $61879,$ 
 $61909,$ 
 $62581,$ 
 $62791,$ 
 $62929,$ 
 $63799,$ 
 $64091,$ 
 $65449,$ 
 $65519,$ 
 $65701,$ 
 $66221,$ 
 $66571,$ 
 $66821,$ 
 $67339,$ 
 $67759,$ 
 $67829,$ 
 $68881,$ 
 $69959,$ 
 $70379,$ 
 $70489,$ 
 $71059,$ 
 $71161,$ 
 $72269,$ 
 $73039,$ 
 $73529,$ 
 $74101,$ 
 $75011,$ 
 $75109,$ 
 $75991,$ 
 $76231,$ 
 $76259,$ 
 $76649,$ 
 $77141,$ 
 $77351,$ 
 $77419,$ 
 $78079,$ 
 $78121,$ 
 $78539,$ 
 $78541,$ 
 $78779,$ 
 $80599,$ 
 $80989,$ 
 $81509,$ 
 $81619,$ 
 $81929,$ 
 $82279,$ 
 $83579,$ 
 $83621,$ 
 $84239,$ 
 $84391,$ 
 $84421,$ 
 $84589,$ 
 $84811,$ 
 $85331,$ 
 $85469,$ 
 $85931,$ 
 $88661,$ 
 $88969,$ 
 $89909,$ 
 $90089,$ 
 $90271,$ 
 $90481,$ 
 $91631,$ 
 $91909,$ 
 $93059,$ 
 $93479,$ 
 $93941,$ 
 $94249,$ 
 $95369,$ 
 $96461,$ 
 $97579,$ 
 $100129,$ 
 $102101,$ 
 $102409,$ 
 $102829,$ 
 $103291,$ 
 $104651,$ 
 $104831,$ 
 $105071,$ 
 $105379,$ 
 $106261,$ 
 $106721,$ 
 $106861,$ 
 $107339,$ 
 $108109,$ 
 $108289,$ 
 $110629,$ 
 $111229,$ 
 $111341,$ 
 $111539,$ 
 $112111,$ 
 $114269,$ 
 $114311,$ 
 $114479,$ 
 $115361,$ 
 $115499,$ 
 $116089,$ 
 $116381,$ 
 $116689,$ 
 $117571,$ 
 $117809,$ 
 $118691,$ 
 $118931,$ 
 $119659,$ 
 $120121,$ 
 $120889,$ 
 $122849,$ 
 $123311,$ 
 $123551,$ 
 $125399,$ 
 $128591,$ 
 $131671,$ 
 $132329,$ 
 $132859,$ 
 $133979,$ 
 $133981,$ 
 $134639,$ 
 $135409,$ 
 $135829,$ 
 $136709,$ 
 $137941,$ 
 $138139,$ 
 $141371,$ 
 $141679,$ 
 $142969,$ 
 $143261,$ 
 $144299,$ 
 $145991,$ 
 $146299,$ 
 $146719,$ 
 $147629,$ 
 $148961,$ 
 $149731,$ 
 $150151,$ 
 $151579,$ 
 $153271,$ 
 $154699,$ 
 $154769,$ 
 $158201,$ 
 $158269,$ 
 $158731,$ 
 $161461,$ 
 $162889,$ 
 $164429,$ 
 $166739,$ 
 $166781,$ 
 $167441,$ 
 $167621,$ 
 $169049,$ 
 $169709,$ 
 $172171,$ 
 $173909,$ 
 $174019,$ 
 $174329,$ 
 $174901,$ 
 $175561,$ 
 $176021,$ 
 $178639,$ 
 $180181,$ 
 $180949,$ 
 $181609,$ 
 $183611,$ 
 $183919,$ 
 $184211,$ 
 $188189,$ 
 $195131,$ 
 $195161,$ 
 $195229,$ 
 $195469,$ 
 $196769,$ 
 $197539,$ 
 $197779,$ 
 $199081,$ 
 $200201,$ 
 $201629,$ 
 $202201,$ 
 $203321,$ 
 $204359,$ 
 $204931,$ 
 $206051,$ 
 $206779,$ 
 $207481,$ 
 $208319,$ 
 $211639,$ 
 $213641,$ 
 $214369,$ 
 $216061,$ 
 $216371,$ 
 $217559,$ 
 $224069,$ 
 $225149,$ 
 $225499,$ 
 $231419,$ 
 $234499,$ 
 $236209,$ 
 $239539,$ 
 $242971,$ 
 $244399,$ 
 $244529,$ 
 $245519,$ 
 $251159,$ 
 $259531,$ 
 $262261,$ 
 $266111,$ 
 $268841,$ 
 $269179,$ 
 $276079,$ 
 $286859,$ 
 $298451,$ 
 $298759,$ 
 $303029,$ 
 $303689,$ 
 $304151,$ 
 $307189,$ 
 $314159,$ 
 $315589,$ 
 $316471,$ 
 $318319,$ 
 $325051,$ 
 $326369,$ 
 $328901,$ 
 $336181,$ 
 $336491,$ 
 $339151,$ 
 $340339,$ 
 $347621,$ 
 $361789,$ 
 $366211,$ 
 $366521,$ 
 $372371,$ 
 $374681,$ 
 $410411,$ 
 $415141,$ 
 $435709,$ 
 $483209,$ 
 $609179,$ 
 $614041,$ 
 $620311,$ 
 $647219,$ 
 $650761,$ 
 $690689,$ 
 $786829,$ 
 $1044889,$ 
 $1624349,$ 
 $1729001,$ 
 $3847271.$
}

\newpage 

\section{Algorithm for $q$ odd and $m=3$}\label{algo-prim-nor}
%{\scriptsize
%\IncMargin{1em}
%\begin{algorithm}
%\KwIn{A prime power $q$ and a positive integer $n \ge 2$}
%\KwOut{List of values of $\beta\in \mathbb{F}_q$ for which there is no $3$-terms in arithmetic progression with desired property}
%$a \gets $ primitive element of $\mathbb{F}_{q^n}$\\
%$List \gets $ empty set\\
%$t \gets \sum_{i=0}^{n-1}q^i$\\
%\For{$j = 0$ \KwTo $\frac{q-1}{2}-1$}
%{
%	$\beta \gets a^{tj}$ \\
%	$R \gets $false\\
%	$u \gets 1$\\
%	\While{$u<q^n$ and $R$ is false}
%		{
%		\If{$\gcd(u,q^n-1)=1$}
%			{$b\gets a^u$\\
%			\If{$b+\beta\neq 0$ and $b+2\beta\neq 0$}
%				{$m_1\gets$ multiplicative order of $b+\beta$\\ 
%				$m_2\gets$ multiplicative order of $b+2\beta$\\
%				\If{$m_1=q^n-1$ and $m_2=q^n-1$}
%					{\uIf{$b$ is normal}{$R\gets$ true}
%					\Else{
%						\uIf{$b+\beta$ is normal}{$R\gets$ true}
%						\Else{
%							\If{$b+2\beta$ is normal}{$R\gets$ true}
%							}
%						}
%					}
%				}
%			}
%		$u \gets u+1$
%		}
%    \If{$R$ is false}{append $\beta$ to $List$}
%}
%\Return $List$
%\caption{Algorithm for $q$ odd and $m=3$}\label{alg:two}
%\end{algorithm}
%}
{\scriptsize
	\IncMargin{1em}
	\begin{algorithm}
		\KwIn{A prime power $q$ and a positive integer $n \ge 2$}
		\KwOut{The list of values $\beta\in \mathbb{F}_q$ for which there is no $3$-terms in arithmetic progression of common difference $\pm \beta$ with the desired property}%{Half of the list of values $\beta\in \mathbb{F}_q$ for which there is no $3$-terms in arithmetic progression with the desired property}
		$a \gets $ primitive element of $\mathbb{F}_{q^n}$\\
		$List \gets $ empty set\\
		$t \gets \sum_{i=0}^{n-1}q^i$\\
		\For{$j = 0$ \KwTo $\frac{q-1}{2}-1$\nllabel{L4}}
		{
			$\beta \gets a^{tj}$ \\
			$R \gets $false\\
			$u \gets 1$\\
			\While{$u<q^n$ and $R$ is false}
			{
				\If{$\gcd(u,q^n-1)=1$}
				{$b\gets a^u$\\
					\If{$b+\beta\neq 0$ and $b+2\beta\neq 0$\nllabel{L11}}
					{$m_1\gets$ multiplicative order of $b+\beta$\\ 
						$m_2\gets$ multiplicative order of $b+2\beta$\nllabel{L13}\\
						\If{$m_1=q^n-1$ and $m_2=q^n-1$\nllabel{L14}}
						{\If{$b$ is normal or $b+\beta$ is normal or $b+2\beta$ is normal\nllabel{L15}}{$R\gets$ true}
							\nllabel{L17}
						}
					}
				}
				$u \gets u+1$
				}
			\If{$R$ is false}{append $\beta$ to $List$}
		}
		\Return $List$
		\caption{Algorithm for $q$ odd and $m=3$}\label{alg:two}
	\end{algorithm}
}

%{\scriptsize
%	\IncMargin{1em}
%	\begin{algorithm}
%		\KwData{A prime power $q$ and a positive integer $n \ge 2$}
%		\KwResult{List of values of $\beta\in \mathbb{F}_q$ for which there is no $3$-terms in arithmetic progression with desired property}
%		$a \gets $ primitive element of $\mathbb{F}_{q^n}$\\
%		$List \gets $ empty set\\
%		$t \gets \sum_{i=0}^{n-1}q^i$\\
%		\For{$j = 0$ \KwTo $\frac{q-1}{2}-1$}
%		{
%			$\beta \gets a^{tj}$ \\
%			$R \gets $false\\
%			$u \gets 1$\\
%			\While{$u<q^n$ and $R$ is false}
%			{
%				\If{$\gcd(u,q^n-1)=1$}
%				{$b\gets a^u$\\
%					\If{$b+\beta\neq 0$ and $b+2\beta\neq 0$}
%					{$m_1\gets$ multiplicative order of $b+\beta$\\ 
%						$m_2\gets$ multiplicative order of $b+2\beta$\\
%						\If{$m_1=q^n-1$ and $m_2=q^n-1$}
%						{$pol\gets \gcd \left(\displaystyle \sum_{i=0}^{n-1} b^{q^i}x^{n-1-i},x^n-1\right)$\\
%							\uIf{$pol=1$}{$R\gets$ true}
%							\Else{
%								$pol\gets \gcd \left(\displaystyle \sum_{i=0}^{n-1} (b+\beta)^{q^i}x^{n-1-i},x^n-1\right) $\\
%								\uIf{$pol=1$}{$R\gets$ true}
%								\Else{
%									$pol\gets \gcd \left(\displaystyle \sum_{i=0}^{n-1} (b+2\beta)^{q^i}x^{n-1-i},x^n-1\right) $\\
%									\If{$pol=1$}{$R\gets$ true}
%								}
%							}
%						}
%					}
%				}
%				$u \gets u+1$
%			}
%			\If{$R$ is false}{append $\beta$ to $List$}
%		}
%		\Return $List$
%		\caption{An algorithm with caption}\label{alg:two}
%	\end{algorithm}
%}

\end{document}